\documentclass{amsart}

\usepackage[utf8]{inputenc}
\usepackage[T1]{fontenc}
\usepackage[a4paper,margin=1in]{geometry}
\usepackage{amsmath,amssymb,amsfonts,amsthm,mathtools}
\usepackage{mathrsfs}
\usepackage{bm}
\usepackage{enumitem}
\usepackage{hyperref}
\usepackage{cite}
\usepackage{comment}
\usepackage{tikz}
\numberwithin{equation}{section}
\usetikzlibrary{arrows.meta,calc}

\hypersetup{
  colorlinks=true,
  linkcolor=blue,
  citecolor=blue,
  urlcolor=blue
}

\newtheorem{theorem}{Theorem}[section]
\newtheorem{proposition}[theorem]{Proposition}
\newtheorem{lemma}[theorem]{Lemma}
\newtheorem{corollary}[theorem]{Corollary}

\newtheorem{remark}[theorem]{Remark}

\newcommand{\G}{\mathcal G}
\newcommand{\K}{\mathcal K}
\newcommand{\R}{\mathbb R}
\newcommand{\C}{\mathbb C}
\newcommand{\D}{\mathscr D}
\newcommand{\N}{\mathbb N}
\newcommand{\I}{\mathcal I}
\newcommand{\ML}{\mathcal L}

\newcommand{\dom}{\operatorname{dom}}
\newcommand{\spec}{\sigma}

\newcommand{\dd}{\,\mathrm d}

\newcommand{\Span}{\operatorname{span}}

\title{ Multiplicity and Nonrelativistic limit of
Bound States of Nonlinear Dirac Equations on Noncompact Metric Graphs with Localized Nonlinearities}
\author{Pan Chen$^{1}$, Qi Guo$^{2}$  }
\date{}

\begin{document}
\maketitle

\begin{center}

{\small $^1$ School of Mathematical Sciences, Shanghai Jiao Tong University, Shanghai, 200240, P. R. China       }

{\small $^{2}$ School of Mathematics, Renmin University of China, Beijing,  100872, P. R. China  }

\end{center}
\noindent{ \bf Abstract.}
In this paper, we investigate the multiplicity of normalized
solutions to a nonlinear Dirac equation with localized 
nonlinearities on noncompact metric graphs under the
 \(L^2\)-constraint, as well as the asymptotic 
 behavior of these solutions in the nonrelativistic
  limit. First, we establish the existence of multiple 
  normalized bound states. Moreover, we explore the
   nonrelativistic limit and show that, 
   as the speed of light tends to infinity,
    the solutions converge to those of a nonlinear 
    Schrödinger equation. Our results 
    including the mass-subcritical, mass-critical and, in particular, mass-supercritical regimes.



\section{Introduction}
{

In this paper we investigate the existence and multiplicity of bound states
of prescribed mass for nonlinear Dirac equations on noncompact metric graphs
with localized nonlinearities. More precisely, we consider
\begin{equation}\label{dirac}
\begin{cases}
\mathscr D_c u-\omega u=a\chi_{\mathcal K}|u|^{p-2}u,
\\[1mm]
\displaystyle \int_{\mathcal G}|u|^2\,dx=1,
\end{cases}
\end{equation}
where \(p\in(2,6)\), \(u:\mathcal G\to\mathbb C^2\), \(a>0\), and
\(\omega\in\mathbb R\) is an unknown Lagrange multiplier. Here
\(\mathcal G\) is a connected noncompact metric graph with finitely many
edges and vertices, \(\mathcal K\) denotes its compact core, and
\(\chi_{\mathcal K}\) is the characteristic function of \(\mathcal K\).
The parameter \(c>0\) is the speed of light and \(m>0\) is the mass of the
particle.

Metric graphs provide one-dimensional models for wave propagation on
branched structures and quantum networks. A metric graph
\(\mathcal G=(V,E)\) is obtained by identifying each bounded edge with a
compact interval \(I_e=[0,\ell_e]\) and each unbounded edge with a half-line
\(I_e=[0,+\infty)\), and then gluing endpoints according to the underlying
combinatorial graph. The metric is the shortest-path distance along the
edges. The compact core \(\mathcal K\) is the metric subgraph consisting of
all bounded edges.

Nonlinear Schrödinger equations on
metric graphs have been extensively studied, both for nonlinearities acting
on the whole graph and for nonlinearities localized on \(\mathcal K\).
For the localized prescribed-mass NLS problem one considers the functional
\begin{equation*}\label{eq:intro-nls-energy}
\mathcal{I}_\infty(f)
=
\frac{1}{2m}\int_{\mathcal G}|f'|^2\,dx
-\frac{2a}{p}\int_{\mathcal K}|f|^p\,dx,
\end{equation*}
constrained on
\[
S_\infty
:=
\left\{
f\in H^1_\mathrm{K}(\mathcal G):
\ \|f\|_{L^2(\mathcal G)}^2=1
\right\},
\]
where \(H^1_{\mathrm{K}}(\mathcal G)\) is the closed subspace of \(H^1(\mathcal G)\) consisting of all functions that are continuous at the vertices.
Critical points of \(\mathcal{I}_\infty\) restricted on ${S_\infty}$ correspond to weak solutions of 
\begin{equation*}\label{eq:intro-nls-el}
-\frac{1}{2m}f''+\nu f
=
a\chi_{\mathcal K}|f|^{p-2}f,
\end{equation*}
where \(\nu\in\mathbb R\) is the Lagrange
multiplier.
For the
existence, nonexistence, and multiplicity results for normalized nonlinear Schrödinger equations
on graphs, including critical and supercritical regimes, can be found in
\cite{AdamiSerraTilliCalcVar2015,AdamiSerraTilliJFA2016,AdamiSerraTilliCMP2017,AdamiSerraTilliCalcVar2019,BorthwickChangJeanjeanSoaveNonlinearity2023, DovettaJDE2018, DovettaSerraTilliAdvMath2020,SerraTentarelliJDE2016,TentarelliJMAA2016}.

A relativistic counterpart of the Schrödinger theory is obtained by
replacing the Laplacian with the Dirac operator. On a metric graph, the
one-dimensional free Dirac operator with mass \(m>0\) and speed of light
\(c>0\) acts edgewise as
\begin{equation*}\label{eq:intro-dirac-edge}
(\mathscr D_c u)_e
=
-ic\sigma_1u_e'
+mc^2\sigma_3u_e,
\qquad
\sigma_1=
\begin{pmatrix}
0&1\\
1&0
\end{pmatrix},
\quad
\sigma_3=
\begin{pmatrix}
1&0\\
0&-1
\end{pmatrix}.
\end{equation*}
The study of Dirac operators and nonlinear Dirac equations on networks is
motivated by effective models in condensed matter physics, nonlinear optics
and field theory. In the graph setting, suitable vertex conditions are needed
in order to obtain a self-adjoint realization of \(\mathscr D_c\). The
Kirchhoff-type conditions used in this paper impose continuity on the first
spinorial component and a balance condition on the second one. 
For the fixed frequency localized nonlinear Dirac equation
\[
\mathscr D_cu-\omega u
=
\chi_{\mathcal K}|u|^{p-2}u,
\qquad
\omega\in(-mc^2,mc^2),
\]
Borrelli, Carlone and Tentarelli proved self-adjointness and spectral
properties of the graph Dirac operator, established the existence and
multiplicity of bound states, and studied their convergence to NLS bound
states as \(c\to+\infty\), see \cite{BorrelliCarloneTentarelliSIAM2019}. Further results on nonlinear Dirac equations on
graphs, including the Cauchy problem, were obtained in
\cite{BorrelliCarloneTentarelliJDE2021}.
The normalized problem on noncompact metric graphs was recently initiated in  \cite{HeJi2025}, where the authors proved existence results in the mass-subcritical case
and gave sufficient assumptions in the supercritical regimes.
For the Dirac equation with prescribed mass, the
mass critical exponent is no longer 6, which is the mass-critical exponent for one-dimensional Schrödinger equation, but the Dirac exponent \(4\). 
This threshold is dictated by the \(L^2\)-preserving scaling on $\mathbb R$. Indeed, under the scaling 
\[
u_\lambda(x)=\lambda^{1/2}u(\lambda x),
\qquad
\|u_\lambda\|_{L^2}=\|u\|_{L^2},
\]
the leading Dirac kinetic term scales as \(\lambda\), whereas
\[
\int_\R |u_\lambda|^p\,dx
=
\lambda^{\frac p2-1}\int_\R |u|^p\,dx.
\]
Thus the kinetic and nonlinear terms have the same scaling precisely when
\(\frac p2-1=1\), namely when \(p=4\). Consequently \(2<p<4\) is the
mass-subcritical range, \(p=4\) is the mass-critical threshold, and
\(4<p<6\) is mass-supercritical for the Dirac equation while still lying
in the Schrödinger subcritical range relevant for the nonrelativistic limit.

In Euclidean spaces, normalized solutions of nonlinear Dirac equations have
been studied by  variational method; see, for instance,
\cite{ChenGuoYuJGA2026,ChenDingGuoWangCalcVar2024,ChenDingGuo2025,CotiZelatiNolasco2019,CotiZelatiNolasco2025,Nolasco2021}. Many of these arguments rely essentially
on Fourier analysis. The positive and negative spectral projectors of the
free Dirac operator are Fourier multipliers; this allows one to prove their
\(L^r\)-boundedness and to establish refined Gagliardo--Nirenberg type
inequalities by frequency decompositions. On a general metric graph, there
is no global Fourier transform compatible with the topology and the vertex
conditions. Hence one cannot directly transfer the Euclidean \(L^r\)-bounds
for spectral projectors, nor the Fourier proof of the modified
Gagliardo--Nirenberg inequalities needed in the supercritical case.
One of the main analytic points of this paper is to replace these Fourier
arguments by tools intrinsic to the graph: spectral calculus for the graph
Laplacian and for \(|\mathscr D_c|\).

There is also a variational difficulty which has no analogue in the
Schrödinger equation. The functional $\mathcal{I}_\infty(f)$  is bounded
from below on a suitable set, whereas the energy functional for Dirac equation is
strongly indefinite, because the spectrum of \(\mathscr D_c\) has both a
positive and a negative unbounded component. Therefore the constrained
functional is unbounded from above and from below even on the \(L^2\)-sphere.
In the subcritical range this can be handled through a global maximization
with respect to the negative spectral directions,
which reduces the functional to the positive spectral sphere. In the
supercritical range \(4<p<6\), such a reduction is only local: the nonlinear
term can no longer be controlled globally by the kinetic part. We
therefore combine a local reduction with a suitable truncation of the
functional. The minimax levels obtained in this truncated setting are then
shown to remain in the nonrelativistic region where the truncated and the
original functionals coincide.

In this paper, we always assume that \(\mathcal G\) is a connected, noncompact metric graph with finitely many
edges and vertices and a nonempty compact core \(\mathcal K\). Our main existence result is as follows.

\begin{theorem}\label{thm:1.1}
 Let
\(m, a>0\), $p\in (2,6)$.
For \(k\in\mathbb N\),
there exist \( a_k>0\) such that, for every \(a\ge a_k, c\ge c(a)\), problem
\eqref{dirac} possesses at least \(k\) distinct pairs of
normalized solutions
\[
(\pm u_{c,1},\omega_{c,1}),\ldots,(\pm u_{c,k},\omega_{c,k})\in H^{1/2}(\G, \mathbb{C}^2)\times \mathbb{R}.
\]
\end{theorem}
\begin{remark}
The case \(a<0\) can be treated in the same way. Indeed, setting
\(b:=-a>0\) and multiplying the equation by \(-1\), we obtain
\[
        -\mathscr{D}_c u+\omega u=b\chi_K |u|^{p-2}u,
        \qquad \|u\|_{L^2(G)}=1 .
\]
Thus the problem is reduced to the same variational framework for the
negative Dirac operator \(-\mathscr{D}_c\). Its positive spectral space is precisely
the negative spectral space of \(\mathscr{D}_c\), namely
\[
        \mathbf 1_{[mc^2,+\infty)}(-\mathscr{D}_c)Y_c=Y_c^- .
\]
Therefore, applying the preceding reduction and the Krasnosel'skii genus
argument to \(-\mathscr{D}_c\) yields the same multiplicity result, with \(|a|\) in
place of \(a\).
\end{remark}

For convenience, the first and second components of \(u_{c,k}\)
 are denoted by \(f_{c,k}\) and \(g_{c,k}\), respectively.
\begin{theorem}\label{them:1.2}
  Under the assumptions of Theorem \ref{thm:1.1}, for each $k\in \mathbb{N}$,
there exist $f_{\infty, k}\in H^1(\G)$ and  constant $\nu_k>0$,  
satisfying
\begin{align*}
  \begin{cases}
&- \frac{1}{2m} f'' +\nu_k f=a\chi_\K|f|^{p-2}f  , \\
&\|f\|_{L^2}^2 =1,
  \end{cases}
\end{align*}
and, up to a subsequence, 
$$\|f_{c, k}-f_{\infty,k}\|_{H^1}\to 0,  \quad  \|g_{c,k}\|_{H^{1}}=\mathcal{O}(\frac{1}{c}) \ \  \text{as}\ \ c\rightarrow \infty.$$
 \end{theorem}

 The
main point of the proof of Theorem \ref{thm:1.1} and Theorem \ref{them:1.2},
is to combine the genus construction for the limiting 
Schr\"odinger equation with a reduction procedure adapted to the 
energy functional of \eqref{dirac}. Owing to the spectral decomposition
\(
Y_c=Y_c^+\oplus Y_c^-
\)(see Section \ref{sec:prelim} for details),
the quadratic part of the functional has opposite signs on the two
components, and therefore the constrained functional cannot be treated by a
direct minimization on the \(L^2\)-sphere. In the range \(2<p\le4\), the
fractional Gagliardo--Nirenberg inequality on the graph gives a global
control of the localized nonlinear term. This allows us to maximize the
functional along the negative spectral directions,
and hence to reduce the problem to the $L^2$-sphere of \(Y_c^+\). The
minimax levels of the reduced functional are then estimated by projecting
compact minimax sets of the limiting Schr\"odinger equation into
\(Y_c^+\). In this way,
for every $k\in \mathbb{N}$, some
negative sublevels of $\mathcal I_\infty$ have genus at least $k$
 give rise,
for \(c\) large, to critical levels of Dirac functional lying below the threshold \(mc^2\),
where the required compactness can be recovered. The case \(4<p<6\) requires
a different implementation of the same idea. Since the reduced functional is
not globally controlled in this regime, the reduction is first performed only
in a nonrelativistic neighborhood of the $L^2$-sphere of \(Y_c^+\). We then introduce a suitable
cutoff and work with a reduced functional defined on the whole
positive sphere. A key ingredient is the  modified
Gagliardo--Nirenberg inequality for
\[
\Lambda_c:=|\mathscr D_c|-mc^2.
\]
 After applying the minimax argument to the reduced functional, we
derive uniform bounds showing that the obtained critical points actually lie
inside the region where the cutoff is inactive. Consequently, they are
critical points of the original functional and solve
\eqref{dirac}.

This paper is organized as follows. Section~2 contains the
preliminary results needed in the variational setting, 
such as genus theory, minimax theorem and
Gagliardo-Nirenberg inequality. Section~3 is
devoted to the mass-subcritical or critical case \(2<p\le4\). There we construct the
global maximization map, study the corresponding reduced functional on
\(S_c^+\),  obtain the first part of the multiplicity result. Section~4 treats the
mass-supercritical range \(4<p<6\). We prove the modified
Gagliardo--Nirenberg inequality near the spectral threshold, introduce the
local reduction and the truncated reduced functional, and carry out the
corresponding minimax construction.  Finally, Section~5 studies the behavior of the
solutions as \(c\to+\infty\). 

}

\section{Preliminary results}\label{sec:prelim}

In this section, we present some preliminary notions
 on the Dirac operator and some basic results about genus theory, minimax theorem and Gagliardo-Nirenberg inequality,
   which will be used in the proofs of the main theorems.

Let \(\mathcal{G}=(V , E )\) be a connected noncompact metric graph with finitely many edges and vertices, and with nonempty compact core \(\mathcal{K}\).  The set of vertices belonging to \(\mathcal{K}\) is denoted by \(V _{\mathcal{K}}\), and the set of bounded edges by \( E _{\mathcal{K}}\).  Every bounded edge \(e\in E _{\mathcal{K}}\) is identified with an interval \(I_e=[0,\ell_e]\), while every unbounded edge is identified with a half-line \([0,+\infty)\).  We fix once and for all an orientation of each edge.  If \(e\) is incident with a vertex \(v\), we write \(e\succ v\).  The initial and terminal endpoints of an oriented bounded edge are denoted by \(e_-\) and \(e_+\), respectively.

For \(e\succ v\), set
\[
\varepsilon_{e,v}:=
\begin{cases}
+1, & v=e_-,\\
-1, & v=e_+.
\end{cases}
\]
Thus, for a scalar function \(h\),
\[
h_e(v)^{\pm}:=\varepsilon_{e,v}h_e(v),
\qquad
\partial_e h(v):=\varepsilon_{e,v}h'_e(v),
\]
where \(\partial_eh(v)\) is the outgoing derivative at \(v\) along \(e\).  If \(e\) is a loop, its two incidences at the same vertex are counted separately, with their corresponding signs.

 Throughout the paper, for $s>0$, 
 we use the edgewise Sobolev space
\[
H^s(\mathcal G,\mathbb C^2)
:=
\bigoplus_{e\in E}H^s(I_e,\mathbb C^2),
\]
endowed with its natural norm, the vertex conditions are not included in
this definition. We also define the following two subspaces of $H^1(\G)$ by
\[
H^1_{\mathrm{K}}(\mathcal{G})
:=
\big\{f\in H^1(\mathcal{G}):
 f_e(v)=f_h(v)\text{ for all }e,h\succ v,\ v\in V _{\mathcal{K}}
\big\},
\]
and
\[
H^1_{\mathrm{D}}(\mathcal{G})
:=
\left\{g\in H^1(\mathcal{G}):
\sum_{e\succ v} g_e(v)^{\pm}=0
\text{ for every }v\in V _{\mathcal{K}}
\right\}.
\]
Let
\[
T:H^1_{\mathrm K}(\mathcal G)\subset L^2(\mathcal G)
\longrightarrow L^2(\mathcal G),
\qquad
Tf=f',
\]
be the first-order derivative operator. Then
\[
\operatorname{dom}(T)=H^1_{\mathrm K}(\mathcal G),
\qquad
\operatorname{dom}(T^*)=H^1_{\mathrm D}(\mathcal G).
\]
We define 
\[
\mathcal L_{\mathcal G}^{\mathrm K}:=T^*T,
\qquad
\mathcal L_{\mathcal G}^{\mathrm D}:=TT^*,
\]
with domain
\[
\operatorname{dom}(L_{\mathcal{G}}^{\mathrm{K}})
=
\left\{
 f\in H^2(\mathcal{G}):\,f_e(v)=f_h(v) \,(e,h\succ v),\,
 \sum_{e\succ v}\partial_e f(v)=0,\,\,v\in V  _{\mathcal{K}}
\right\},
\]

\[
\operatorname{dom}(L_{\mathcal{G }}^{\mathrm{D}})
=
\left\{
 g\in H^2(\mathcal{G}):
g'_e(v)=g'_h(v)\, ( e,h\succ v),\,\sum_{e\succ v}g_e(v)^{\pm}=0,\,\, v\in V  _{\mathcal{K}}
\right\}.
\]
They act edgewise as \(-d^2/dx^2\). Their form domains are
\[
\operatorname{dom}\big((\mathcal L_{\mathcal G}^{\mathrm K})^{1/2}\big)
=
H^1_{\mathrm K}(\mathcal G),
\qquad
\operatorname{dom}\big((\mathcal L_{\mathcal G}^{\mathrm D})^{1/2}\big)
=
H^1_{\mathrm D}(\mathcal G),
\]
The self-adjointness, spectra and
kernels of these two operators are given in Appendix~\ref{app:graph-operators}.

The massless Dirac operator is written as
\[
\mathscr D_0
=
\begin{pmatrix}
0&iT^*\\
-iT&0
\end{pmatrix},
\qquad
\operatorname{dom}(\mathscr D_0)
=
H^1_{\mathrm K}(\mathcal G)
\oplus
H^1_{\mathrm D}(\mathcal G).
\]
Here and below, the direct sum is the external direct sum of the two spinorial
components. For \(c>0\) and \(m>0\), the massive Dirac operator is
\[
\mathscr D_c
=
c\mathscr D_0+mc^2\sigma_3,
\qquad
\operatorname{dom}(\mathscr D_c)=\operatorname{dom}(\mathscr D_0).
\]
Equivalently, \(\mathscr D_c\) acts edgewise as
\[
(\mathscr D_c u)_e
=
-ic\sigma_1u'_e+mc^2\sigma_3u_e.
\]
Moreover,
\[
\mathscr D_0^2
=
\mathcal L_{\mathcal G}^{\mathrm K}
\oplus
\mathcal L_{\mathcal G}^{\mathrm D},
\qquad
\operatorname{dom}(\mathscr D_0^2)
=
\operatorname{dom}(\mathcal L_{\mathcal G}^{\mathrm K})
\oplus
\operatorname{dom}(\mathcal L_{\mathcal G}^{\mathrm D}),
\]
and hence
\[
\mathscr D_c^2
=
c^2\mathscr D_0^2+m^2c^4I.
\]
In particular,
\[
|\mathscr D_c|
=
\big(c^2\mathscr D_0^2+m^2c^4I\big)^{1/2}
=
\big(c^2\mathcal L_{\mathcal G}^{\mathrm K}+m^2c^4\big)^{1/2}
\oplus
\big(c^2\mathcal L_{\mathcal G}^{\mathrm D}+m^2c^4\big)^{1/2}.
\]
with domain
\[
\operatorname{dom}(|\mathscr D_c|)
=
\operatorname{dom}(\mathscr D_c).
\]
The spectral of \(\mathscr D_c\) is 
\begin{equation}\label{spec}
     \sigma(\mathscr D_c)
=
(-\infty,-mc^2]\cup[mc^2,+\infty),   
\end{equation}
For more details, we refer the reader to \cite{BorrelliCarloneTentarelliSIAM2019,BorrelliCarloneTentarelliNote2019} and to Appendix~\ref{app:graph-operators} as well.
The form domain of $\mathscr{D}_c$ is defined as
\begin{equation}\label{eq:intro-form-domain}
Y_c
:=
\operatorname{dom}(|\mathscr D_c|^{1/2})
=
\big[L^2(\mathcal G,\mathbb C^2),
\operatorname{dom}(\mathscr D_c)\big]_{\frac12}.
\end{equation}
It is a closed subspace of 
\[H^{1/2}(\G,\C^2)=\big[L^2(\G,\C^2),H^{1}(\G,\C^2)\big]_{\frac12}.\] Hence, for
every \(2\le r<+\infty\),
\begin{equation}\label{eq:intro-embedding}
Y_c\hookrightarrow L^r(G,\mathbb C^2)\ \text{continuously},\qquad
Y_c\hookrightarrow L^r(K,\mathbb C^2)\ \text{compactly}.
\end{equation}
We equip \(Y_c\) with the inner product
\[
(u,v)_c
:=
\Re 
\big(|\mathscr D_c|^{1/2}u,|\mathscr D_c|^{1/2}v\big)_{L^2},
\qquad
\|u\|_c^2=(u,u)_c.
\]
The spectral gap \eqref{spec} gives
\[
\|u\|_c^2\ge mc^2\|u\|_{L^2}^2,
\qquad u\in Y_c.
\]
As a vector space, \(Y_c\) is independent of \(c\). More precisely,
\[
Y_c
=
\operatorname{dom}\big((I+\mathscr D_0^2)^{1/4}\big)
=
\operatorname{dom}\big((I+\mathcal L_{\mathcal G}^{\mathrm K})^{1/4}\big)
\oplus
\operatorname{dom}\big((I+\mathcal L_{\mathcal G}^{\mathrm D})^{1/4}\big).
\]
Since $Y_c=Y_1$ is a closed subspace of $H^{1/2}(\G,\mathbb{C}^2)$, then
\[
C^{-1}\|u\|_{H^{1/2}}\leq \|u\|_1\leq C \|u\|_{H^{1/2}}.
\]
Moreover, as stated in Appendix~\ref{app:graph-operators}, for each \(c>1\),
\begin{equation}\label{3}
         C^{-1}c\|u\|_{H^{1/2}}^2 \le \|u\|_c^2 \le Cc^2\|u\|_{H^{1/2}}^2.
\end{equation}
Let $P_c^\pm$ be the projectors onto positive or negative spectral subspaces for $\mathscr{D}_c$, that is
\[
P_c^+
=
\mathbf 1_{[mc^2,+\infty)}(\mathscr D_c),
\qquad
P_c^-
=
\mathbf 1_{(-\infty,-mc^2]}(\mathscr D_c).
\]
Define
\[
L_c^\pm:=P_c^\pm L^2(\mathcal G,\mathbb C^2),\qquad Y_c^\pm:=P_c^\pm Y_c.
\]
Then the direct sum decompositions
\[
L^2(\mathcal G,\mathbb C^2)=L_c^+\oplus L_c^-,\qquad Y_c=Y_c^+\oplus Y_c^-
\]
hold.
 Let
\begin{equation}\label{eq:spheres}
S_c:=\{u\in Y_c:\|u\|_{L^2}^2=1\},\qquad S_c^\pm:=S_c\cap Y_c^\pm.
\end{equation}
For \(u=u^++u^-\in Y_c^+\oplus Y_c^-\), define
\[
\mathcal A(u):=\frac{2a}{p}\int_{\mathcal K} |u|^p\,dx.
\]
One can easily see that a bound state of \eqref{dirac}  is a
critical point of
\begin{equation}\label{eq:Ic}
\mathcal I_c(u)
:=\|u^+\|_c^2-\|u^-\|_c^2
-\frac{2a}{p}\int_\K |u|^p\dd x.
\end{equation}
The derivative, as a real functional, is
\begin{equation}\label{eq:IcDerivative}
\dd\mathcal I_c(u)[\varphi]
=2(u^+,\varphi^+)_c-2(u^-,\varphi^-)_c
-2a\Re \int_\K |u|^{p-2}u\cdot\overline{\varphi}\dd x.
\end{equation}
A critical point of $\mathcal I_c|_{S_c}$ is a weak solution of \eqref{dirac}; more precisely, there exists $\omega\in\R$ such that
\begin{equation}\label{eq:weakDirac}
\dd\mathcal I_c(u)[\varphi]=2\omega\Re \int_\G u\cdot\overline\varphi\dd x,
\qquad \varphi\in Y_c.
\end{equation}

The following lemma can be easily proved for the Euclidean space case by means of the Fourier analysis,
 here we use spectral measures to prove the version on metric graphs.
\begin{lemma}\label{lemma2.1}
There exists a constant \(C>0\), independent of \(c>1\), such that
\[
 \|u\|_c^2-mc^2\|u\|_{L^2(\mathcal G)}^2
 +\|u\|_{L^2(\mathcal G)}^2
 \ge C\|u\|_1^2 ,
 \qquad u\in Y_c.
\]
\end{lemma}

\begin{proof}
Let
\[
 A:=\mathscr D_0^2 .
\]
Then 
\[
 |\mathscr D_c|=(c^2A+m^2c^4I)^{1/2}.
\]

Let \(E_A(\cdot)\) be the spectral resolution of \(A\). For \(u\in Y_c\), define 
\[
 d\mu_u(\lambda):=d\|E_A(\lambda)u\|_{L^2(\mathcal G)}^2 .
\]
By the spectral theorem,
\[
 \|u\|_c^2
 =
 \int_{[0,\infty)}
 (c^2\lambda+m^2c^4)^{1/2}\,d\mu_u(\lambda),
\]
and, since \(\|\cdot\|_1\) is the form norm associated with
\(|\mathscr D_1|=(A+m^2I)^{1/2}\),
\[
 \|u\|_1^2
 =
 \int_{[0,\infty)}
 (\lambda+m^2)^{1/2}\,d\mu_u(\lambda).
\]
Moreover,
\[
 \|u\|_{L^2(\mathcal G)}^2
 =
 \int_{[0,\infty)}d\mu_u(\lambda).
\]
Hence
\[
\begin{aligned}
&\|u\|_c^2-mc^2\|u\|_{L^2(\mathcal G)}^2
 +\|u\|_{L^2(\mathcal G)}^2  \\
&\qquad =
 \int_{[0,\infty)}
 \left[
 (c^2\lambda+m^2c^4)^{1/2}-mc^2+1
 \right]\,d\mu_u(\lambda).
\end{aligned}
\]

A direct computation yields,
for \(\lambda\ge0\) and \(c>1\), 
\[
 (c^2\lambda+m^2c^4)^{1/2}-mc^2+1
 =
 \frac{\lambda}{
 \left(m^2+\frac{\lambda}{c^2}\right)^{1/2}+m
 } +1 \geq C(\lambda+m^2)^{1/2},
\]
Integrating this inequality with respect to \(d\mu_u\), we obtain
\[
\begin{aligned}
&\|u\|_c^2-mc^2\|u\|_{L^2(\mathcal G)}^2
 +\|u\|_{L^2(\mathcal G)}^2  \\
&\qquad \ge
 C_m\int_{[0,\infty)}(\lambda+m^2)^{1/2}\,d\mu_u(\lambda)
 =
 C_m\|u\|_1^2 .
\end{aligned}
\]
The proof is complete.
\end{proof}

We shall use the following Gagliardo--Nirenberg inequality on
noncompact metric graphs. 

\begin{lemma}
\label{lem:graph_Hhalf_GN}
Let \(\G\) be a connected noncompact metric graph with finitely many edges and
vertices. For every \(q\in[2,+\infty)\) there exists \(C_q>0\), depending only
on \(q\) and on \(\G\), such that
\begin{enumerate}[label=\textup{(\roman*)}]
\item For every \(f\in H^1(\G)\),
\begin{equation}
\label{eq:GN_H1_graph}
\|f\|_{L^q(\G)}^q
 \le C_q
 \|f\|_{L^2(\G)}^{\frac q2+1}
 \|f'\|_{L^2(\G)}^{\frac q2-1}.
\end{equation}
\item For $u\in H^{1/2}(\G)$, 
\begin{equation}
\label{eq:new}
\|u\|_{L^q(\G)}^q
 \le C_q
 \|u\|_{L^2(\G)}^2
 \|u\|_{H^{1/2}(\G)}^{q-2},
\end{equation}
\item For $f\in \operatorname{dom}\big((L_{\mathcal{G}}^{\mathrm{K}})^{1/4}\big)$, there holds
\begin{equation}
\label{eq:Hhalf_GN_graph}
\|f\|_{L^q(\G)}^q
 \le C_q
 \|f\|_{L^2(\G)}^2
 \|(L_\G^{\mathrm{K}})^{1/4}f\|_{L^2(\G)}^{q-2}.
\end{equation}
\end{enumerate}
\end{lemma}

\begin{proof}
        \begin{enumerate}[label=\textup{(\roman*)}]\item
  The proof of \eqref{eq:GN_H1_graph} can be found in \cite{AdamiSerraTilliJFA2016}.
\item 
Set
\[
s:=\frac12-\frac1q\in\left(0,\frac12\right),
\qquad
\theta:=2s=1-\frac2q.
\]
Then
\[
1-\theta=\frac2q.
\]

The Sobolev embedding
\(
H^s(\G)\hookrightarrow L^q(\G),
\)
gives
\begin{equation}\label{1}
        \|u\|_{L^q(\G)}
\le
C\|u\|_{H^s(\G)}.
\end{equation}
It follows from the interpolation 
\[
H^s(\G)
=
\bigl[L^2(\G),H^{1/2}(\G)\bigr]_{\theta},
\qquad
\theta=2s
\]
we get
\begin{equation}\label{2}
  \|u\|_{H^s(\G)}
\le
C
\|u\|_{L^2(\G)}^{1-\theta}
\|u\|_{H^{1/2}(\G)}^{\theta}.      
\end{equation}
Combining \eqref{1} with \eqref{2} gives
\[
\|u\|_{L^q(\G)}
\le
C
\|u\|_{L^2(\G)}^{1-\theta}
\|u\|_{H^{1/2}(\G)}^{\theta}.
\]
This proves the desired inequality.
  \item 
We first prove 
 the fractional Gagliardo--Nirenberg inequality on the half-line.
Let \(L_+\) be the Laplacian operator on \(\mathbb R^+\) with domain
\[
\dom(L_+)=\{g\in H^2(\mathbb R^+):g'(0)=0 \}.
\] 
For every \(g\in \dom(L_+^{1/4})\), let \(Eg\) be the even extension of \(g\) to \(\mathbb R\):
\[
(Eg)(x):=g(|x|),\qquad x\in\mathbb R.
\]
 Let \(\mathcal C\) be the unitary cosine transform on \(L^2(\mathbb R^+)\), that is
\[
(\mathcal Cg)(\xi)
=
\sqrt{\frac2\pi}
\int_0^\infty g(x)\cos(x\xi)\,dx.
\qquad \xi\ge0.
\]
Then
\[
\mathcal C(L_+g)(\xi)=\xi^2(\mathcal Cg)(\xi),
\]
and
\[
\|L_+^{1/4}g\|_{L^2(\mathbb R^+)}^2
=
\int_0^\infty \xi\,|\mathcal Cg(\xi)|^2\,d\xi.
\]
Using the unitary Fourier transform on \(\mathbb R\), one has
\[
\widehat{Eg}(\xi)=\mathcal Cg(|\xi|)
\qquad\text{for a.e. }\xi\in\mathbb R.
\]
Therefore
\begin{align}
\|(-\Delta)^{1/4}Eg\|_{L^2(\mathbb R)}^2
&=
\int_{\mathbb R}|\xi|\,|\widehat{Eg}(\xi)|^2\,d\xi \notag\\
&=
2\int_0^\infty \xi\,|\mathcal Cg(\xi)|^2\,d\xi \notag\\
&=
2\|L_+^{1/4}g\|_{L^2(\mathbb R^+)}^2 .
\label{eq:even_extension_half_derivative}
\end{align}
Applying  one-dimensional fractional Gagliardo--Nirenberg
inequality 
\begin{equation}
\label{eq:GN_real_line_fractional}
\|h\|_{L^q(\mathbb R)}
\le
C_q
\|h\|_{L^2(\mathbb R)}^{2/q}
\|(-\Delta)^{1/4}h\|_{L^2(\mathbb R)}^{1-2/q}.
\end{equation}
 to \(h=Eg\), and using \eqref{eq:even_extension_half_derivative}, we get
 the fractional Gagliardo--Nirenberg inequality on the half-line:
\begin{align}\label{half-gn}
2\|g\|_{L^q(\mathbb R^+)}^q
&=
\|Eg\|_{L^q(\mathbb R)}^q                                                   \\
&\le
C_q
\|Eg\|_{L^2(\mathbb R)}^2
\|(-\Delta)^{1/4}Eg\|_{L^2(\mathbb R)}^{q-2}                                      \\
&=
C_q
\left(2\|g\|_{L^2(\mathbb R^+)}^2\right)
\left(\sqrt2\,\|L_+^{1/4}g\|_{L^2(\mathbb R^+)}\right)^{q-2}.
\end{align}

Now, we prove the fractional P\'olya--Szeg\H{o} inequality by the quadratic \(K\)-method.
Let \(u\ge0\), \(u\in L^2(\G)\), and let \(u^*\) be its decreasing rearrangement on
\(\mathbb R^+\):
\[
u^*(s)
:=
\inf\{t\ge0:\rho_u(t)\le s\},
\qquad
\rho_u(t)
:=
\operatorname{meas}\{x\in \G:\ u(x)>t\}.
\]
Equimeasurability gives
\begin{equation}
\label{eq:equimeasurability}
\|u^*\|_{L^r(\mathbb R^+)}
=
\|u\|_{L^r(\G)}
\qquad\text{for every }r\in[1,+\infty).
\end{equation}
For \(t>0\) and \(u\in L^2(\G)\), define the quadratic \(K\)-functional 
\[
\mathfrak K_\G(t,u)^2
:=
\inf_{v\in H^1_{\mathrm{K}}(\G)}
\left\{
\|u-v\|_{L^2(\G)}^2
+
t\int_\G |v'(x)|^2\,dx
\right\}.
\]
Similarly, for \(g\in L^2(\mathbb R^+)\), set
\[
\mathfrak K_+(t,g)^2
:=
\inf_{\phi\in H^1_{\mathrm{K}}(\mathbb R^+)}
\left\{
\|g-\phi\|_{L^2(\mathbb R^+)}^2
+
t\int_0^\infty |\phi'(s)|^2\,ds
\right\}.
\]
For convenience, we denote the Laplacian operator on \(L^2(\G)\) with Kirchhoff
conditions.
Let \(E_\G(\lambda)\) be the spectral resolution of \(L_\G^{\mathrm{K}}\), and let
\[
d\mu_u^\G(\lambda)
:=
d\|E_\G(\lambda)u\|_{L^2(\G)}^2.
\]
By the spectral theorem,
\begin{equation}
\label{eq:K_spectral_graph}
\mathfrak K_\G(t,u)^2
=
\int_{[0,+\infty)}
\frac{t\lambda}{1+t\lambda}\,d\mu_u^\G(\lambda).
\end{equation}
Indeed, in the spectral representation of \(L_\G^{\mathrm{K}}\), the minimization becomes
pointwise in the spectral variable. For a scalar spectral component \(z\), one has
\[
\inf_{y\in\mathbb C}
\left\{
|z-y|^2+t\lambda |y|^2
\right\}
=
\frac{t\lambda}{1+t\lambda}|z|^2,
\]
with minimizer \(y=(1+t\lambda)^{-1}z\). Integrating this pointwise identity gives
\eqref{eq:K_spectral_graph}. The same formula holds on \(\mathbb R^+\), with
\(L_+\) and \(\mathfrak K_+\). Using \eqref{eq:K_spectral_graph}, we get
\begin{align}
\int_0^\infty
t^{-1/2}\mathfrak K_\G(t,u)^2\,\frac{dt}{t}
&=
\int_{[0,+\infty)}
\left(
\int_0^\infty
t^{-1/2}\frac{t\lambda}{1+t\lambda}\,\frac{dt}{t}
\right)
d\mu_u^\G(\lambda) \notag\\
&=
\int_{[0,+\infty)}
\left(
\int_0^\infty
\frac{\lambda t^{-1/2}}{1+t\lambda}\,dt
\right)
d\mu_u^\G(\lambda) \notag\\
&=
\pi
\int_{[0,+\infty)}
\lambda^{1/2}\,d\mu_u^\G(\lambda) \notag\\
&=
\pi\|(L_\G^{\mathrm{K}})^{1/4}u\|_{L^2(\G)}^2.
\label{eq:K_identity_graph_Hhalf}
\end{align}
In the third equality we use the fact 
\[
\int_0^\infty \frac{s^{-1/2}}{1+s}\,ds=\pi.
\]
Again, the same identity holds on \(\mathbb R^+\):
\begin{equation}
\label{eq:K_identity_halfline_Hhalf}
\int_0^\infty
t^{-1/2}\mathfrak K_+(t,g)^2\,\frac{dt}{t}
=
\pi\|L_+^{1/4}g\|_{L^2(\mathbb R^+)}^2.
\end{equation}
We now compare \(\mathfrak K_+(t,u^*)\) and \(\mathfrak K_\G(t,u)\). Let
\(v\in H^1(\G)\). Since \(u\ge0\), replacing \(v\) by \(|v|\), we have
\[
\|u-|v|\|_{L^2(\G)}
\le
\|u-v\|_{L^2(\G)},
\qquad
\int_\G |(|v|)'|^2\,dx
\le
\int_\G |v'|^2\,dx.
\]
Thus, in the infimum defining \(\mathfrak K_\G(t,u)\), it is enough to consider
nonnegative \(v\).
It follows from  Hardy--Littlewood rearrangement inequality
\[
\int_\G uv\,dx
\le
\int_0^\infty u^*v^*\,ds,
\]
we obtain 
\begin{equation}
\label{eq:rearrangement_L2_contraction}
\|u^*-v^*\|_{L^2(\mathbb R^+)}
\le
\|u-v\|_{L^2(\G)}.
\end{equation}
Moreover, the P\'olya--Szeg\H{o} inequality for decreasing rearrangements yields
\[
\int_0^\infty |(v^*)'(s)|^2\,ds
\le
\int_\G |v'(x)|^2\,dx.
\]
Consequently, for every nonnegative \(v\in H^1_{\mathrm{K}}(\G)\),
\begin{align*}
\mathfrak K_+(t,u^*)^2
&\le
\|u^*-v^*\|_{L^2(\mathbb R^+)}^2
+
t\int_0^\infty |(v^*)'(s)|^2\,ds                                      \\
&\le
\|u-v\|_{L^2(G)}^2
+
t\int_G |v'(x)|^2\,dx.
\end{align*}
Taking the infimum over \(v\in H^1_{\mathrm{K}}(\G)\), we obtain
\begin{equation}
\label{eq:K_contraction_graph_halfline}
\mathfrak K_+(t,u^*)
\le
\mathfrak K_\G(t,u)
\qquad\text{for every }t>0.
\end{equation}
Combining \eqref{eq:K_contraction_graph_halfline},
\eqref{eq:K_identity_graph_Hhalf}, and \eqref{eq:K_identity_halfline_Hhalf}, we get
the fractional P\'olya--Szeg\H{o} inequality
\begin{equation}
\label{eq:fractional_polya_szego_Hhalf}
\|L_+^{1/4}u^*\|_{L^2(\mathbb R^+)}
\le
\|(L_\G^{\mathrm{K}})^{1/4}u\|_{L^2(\G)}.
\end{equation}
Assume first that \(f\ge0\).
Using the Gagliardo--Nirenberg inequality on the half-line \eqref{half-gn}, with \(g=f^*\), and then
the fractional P\'olya--Szeg\H{o} inequality
\eqref{eq:fractional_polya_szego_Hhalf}, we get
\begin{align*}
\|f\|_{L^q(\G)}^q
&=
\|f^*\|_{L^q(\mathbb R^+)}^q                                      \\
&\le
C_q
\|f^*\|_{L^2(\mathbb R^+)}^2
\|L_+^{1/4}f^*\|_{L^2(\mathbb R^+)}^{q-2}                         \\
&\le
C_q
\|f\|_{L^2(\G)}^2
\|(L_\G^{\mathrm{K}})^{1/4}f\|_{L^2(\G)}^{q-2}.
\end{align*}
Applying the nonnegative case to \(|f|\) proves \eqref{eq:Hhalf_GN_graph}.

\end{enumerate}
\end{proof}

By combining \eqref{3} and \eqref{eq:new}, we have the following results.

\begin{corollary}\label{lem:GNDirac}
Let $q\ge2$ and $c\ge1$. There exists a constant $C_q>0$, independent of $c$, such that 
for every $u\in Y_c$,
\begin{equation}\label{eq:GNDirac}
\int_\G |u|^q\dd x
\le C_{q}\,c^{-\frac{q-2}{2}}\|u\|_c^{q-2}\|u\|_{L^2}^2.
\end{equation}
\end{corollary}

We now recall the notion of genus due to M. A. Krasnosel'skii, for a complete discussion, see
\cite{Krasnoselskii1964,Rabinowitz1986}.
Consider a real Banach space $E$ with norm $\|\cdot\|_E$, and a real Hilbert space $H$ equipped with inner product $(\cdot, \cdot)_H$. Define the manifold
\begin{equation*}
    \mathcal{M} := \{u \in E \mid (u, u)_H = 1\}.
\end{equation*}

Let $\Sigma(\mathcal{M})$ be the collection of all closed and symmetric subsets of $\mathcal{M}$. For any nonempty $A \in \Sigma(\mathcal{M})$, the genus $\gamma(A)$ is the smallest integer $k \geq 1$ for which an odd continuous map $\phi: A \to \mathbb{R}^k \setminus \{0\}$ exists; i.e.,
$$
\gamma(A) := \min\{n \in \mathbb{N} : \exists \phi : A \rightarrow \mathbb{R}^n \setminus \{0\}, \, \varphi \text{ continuous and odd}\}.
$$
If there is no such integer, we set $\gamma(A) = \infty$. For each $k \in \mathbb{N}$, define
\begin{equation*}
    \Gamma_k := \{A \in \Sigma(\mathcal{M}) \mid \gamma(A) \geq k\}.
\end{equation*}

The following minimax theorem, which is taken from \cite[Theorem 2.1]{JeanjeanLu2019}, is presented without proof; the reader is referred to the original source for details.

\begin{theorem}\label{minimax}
Let $\mathcal{I}: E \to \mathbb{R}$ be an even $C^1$ functional. Assume that $I|_\mathcal{M}$ is bounded below and satisfies the (PS)$_d$ condition for every $d \in \Xi\subset \mathbb{R}$. Additionally, suppose that $\Gamma_k \neq \emptyset$ for all $k \in \mathbb{N}$. Then one can define the minimax values $-\infty < \lambda_1 \leq  \lambda_2 \leq  \cdots \leq  \lambda_n\leq \cdots$ by
\begin{equation*}
    \lambda_k := \inf_{A \in \Gamma_k} \sup_{u \in A} \mathcal{I}(u), \quad k \geq  1,
\end{equation*}
and the following statements hold:
\begin{enumerate}[label=\textup{(\roman*)}]
    \item If $\lambda_k \in \Xi$, then $\lambda_k$ is a critical value of $\mathcal{I}|_\mathcal{M}$.
    \item Let $L^\lambda$ denote the set of critical points of $\mathcal{I}|_\mathcal{M}$ at level $\lambda \in \mathbb{R}$. If
    \begin{equation*}
        \lambda_k = \lambda_{k+1} = \cdots = \lambda_{k+l-1} =: \lambda \in \Xi  \quad \text{for some } k, l \geq  1,
    \end{equation*}
    then $\gamma(L^\lambda) \geq  l$. Consequently, $I|_\mathcal{M}$ possesses infinitely many critical points at level $\lambda$ whenever $l\geq  2$.
\end{enumerate}
\end{theorem}

Recall that
\begin{equation}\label{eq:Sinf}
S_\infty=\{f\in H_\mathrm{K}^1(\G):\|f\|_{L^2}^2=1\}
\end{equation}
and 
\begin{equation*}
\mathcal{I}_\infty(f)
=
\frac{1}{2m}\int_{\mathcal G}|f'|^2\,dx
-\frac{2a}{p}\int_{\mathcal K}|f|^p\,dx,
\qquad 2<p.
\end{equation*}
 For $j\in\N$, set
\begin{equation}\label{eq:einfj}
\Gamma_{\infty,j}:=\{A\subset S_\infty:A\text{ compact, symmetric, and }\gamma(A)\ge j\},
\qquad
e_j^a:=\inf_{A\in\Gamma_{\infty,j}}\sup_{f\in A}\mathcal I_\infty(f).
\end{equation}

The following lemma is precisely the negative level 
part of Serra and Tentarelli's genus 
construction for localized NLS bound states, 
see the proof of \cite[Theorem 1.2]{SerraTentarelliJDE2016}.

\begin{lemma}\label{lem:NLSnegative}
Let $2<p<6$. For every $k\in\N$ there exists $a_k>0$ such that, for every $a\ge a_k$,
\begin{equation}\label{eq:NLSnegative}
e_1^a\le e_2^a\le\cdots\le e_k^a<0.
\end{equation}
\end{lemma}
\begin{remark}
  In \cite{SerraTentarelliJDE2016}, the authors consider the functional
\[
E(u)=\frac12\int_\G |u'|^2\dd x-\frac1p\int_\K |u|^p\dd x,
\qquad
\|u\|_{L^2}^2=\mu .
\]
By normalization $f = u/\sqrt{\mu}$, we have
\begin{equation}\label{eq:massCouplingNormalization}
\frac{E(\sqrt\mu f)}{\mu}
=\frac12\int_\G |f'|^2\dd x-\frac{\mu^{(p-2)/2}}{p}\int_\K |f|^p\dd x.
\end{equation}
Consequently, the requirement in Theorem 1.2 of \cite{SerraTentarelliJDE2016} that $\mu$ be sufficiently large is equivalent to the condition that $a$ be sufficiently large in the present paper.
\end{remark}

The following lemma shows the compactness of Palais-Smale sequences with negative energy for $\mathcal I_\infty|_{S_\infty}$, see Proposition 4.4 in \cite{SerraTentarelliJDE2016}.
\begin{lemma}\label{lem:NLSPS}
Let $2<p<6$. The functional $\mathcal I_\infty|_{S_\infty}$ satisfies the Palais--Smale condition at every level $d<0$.
\end{lemma}
Combining Lemma \ref{minimax}, Lemma \ref{lem:NLSnegative}, and Lemma \ref{lem:NLSPS}, one obtains the multiplicity of critical points of the functional $\mathcal I_\infty|_{S_\infty}$, namely Theorem 1.2 in \cite{SerraTentarelliJDE2016}.

\section{Mass-Subcritical or Mass-Critical Case}\label{sec:subcritical}

In this section, we mainly prove the multiplicity 
of solutions to \eqref{dirac} for the mass-subcritical or
critical case, namely, $2<p\leq 4$.

For $v\in S_c^+$, we define 
\begin{equation}\label{eq:fiber}
\mathcal M_c(v):=\bigl(Y_c^-\oplus\Span\{v\}\bigr)\cap S_c
=\{t v+w: t\ge0,\ w\in Y_c^-,\ t^2+\|w\|_{L^2}^2=1\}.
\end{equation}
To overcome the difficulty that $\mathcal{I}_c$ is
 unbounded from below on the \(L^2\)-sphere, 
 we first reduce  $\mathcal{I}_c$ to  \(S_c^+\)
  via the following maximization problem
  \[
  \rho(v)=\max_{u\in\mathcal M_c(v)}\mathcal I_c(u).
  \]
   The reduced functional is in one-to-one 
   correspondence with the functional $\mathcal{I}_c$
    regarding their critical points, and,
     in contrast to $\mathcal{I}_c$, the reduced 
     functional is bounded below, 
     making it more tractable from a 
     variational perspective. 
     We summarize the properties of the reduced
      functional as follows.
      We omit the proof here, as it is very similar 
      to that in the supercritical case, 
      where a detailed proof will be provided 
      (see Proposition \ref{prop:localReductionSuper}).
{

\begin{proposition}\label{prop:reduction}
  There exists $c_0>0$ such that, for every $c\ge c_0$ and every
$v\in S_c^+$, there exists
\[
        \Phi_c(v)\in\mathcal M_c(v)
\]
such that
\[
        \mathcal I_c(\Phi_c(v))=\max_{u\in\mathcal M_c(v)}\mathcal I_c(u)=\rho_c(v).
\]
and 
\[
\dd\I_c(\Phi_c(v))[h]-2\omega\Re((\Phi_c(v), h)_{L^2}=0, \quad \forall h\in Y_c^-\oplus\operatorname{span}\{v\}.
\]
with $\omega=\omega(\Phi_c(v))\in \mathbb{R}^+$.
Writing
\[
        \Phi_c(v)=t_c(v)v+w_c(v),
        \qquad t_c(v)\ge0,
        \qquad w_c(v)\in Y_c^- ,
\]
and the reduced functional is defined as
\[
        \mathcal J_c(v):=\mathcal I_c(\Phi_c(v)),
        \qquad v\in S_c^+,
\]
one has $t_c(v)>0$.  Moreover:
\begin{enumerate}[label=\textup{(\roman*)}]
\item the estimate
\begin{equation}\label{eq:fibre-basic-estimate-final-insert}
        \mathcal A(\Phi_c(v))+\|\Phi_c(v)^-\|_c^2
        +\|\Phi_c(v)^-\|_{L^2}^2\|v\|_c^2
        \le \mathcal A(v)
\end{equation}
holds.  In particular,
\begin{equation}\label{eq:negative-part-bound-final-insert}
        \|w_c(v)\|_c^2\le \mathcal A(v)=\frac{2a}{p}\int_K |v|^p\,dx;
\end{equation}
\item Up to a phase factor, the
maximizer is unique;
\item The map $\Phi_c:S_c^+\to S_c$ and the
reduced functional $ \mathcal J_c(v)$
is of class $C^1$;
\item $u_n=\Phi_c(v_n)$ is a bounded Palais-Smale sequence of
 $\mathcal I_c|_{S_c}$ provided $v_n\in S_c^+$ is a bounded Palais-Smale sequence
  of $\mathcal J_c$ on $S_c^+$. Moreover, up to a subsequence, assume $\omega(v_n)\to \omega$,
  then $u_n=\Phi_c(v_n)$ is a bounded Palais-Smale sequence of the functional
  \[
  \I_c^\omega(u):=\I_c(u)-\omega\|u\|_{L^2}^2.
  \]
\end{enumerate}
\end{proposition}

}

\begin{lemma}\label{lem:FWgraph}
Let \(A\subset S_{\infty}\) be  bounded  in \(H^1_{\mathrm K}(\mathcal G)\), where
\[
S_{\infty}
:=
\left\{
f\in H^1_{\mathrm K}(\mathcal G):
\|f\|_{L^2(\mathcal G)}=1
\right\}.
\]
For \(f\in A\), set
\[
\iota(f):=\binom{f}{0},
\qquad
z_c(f):=P_c^+\iota(f),
\qquad
\Theta_c(f):=\frac{z_c(f)}{\|z_c(f)\|_{L^2}}\in S_c^+.
\]
Then, as \(c\to+\infty\),
\[
\mathcal J_c(\Theta_c(f))
\le
mc^2+\I_{\infty}(f)+o_c(1),
\]
uniformly for \(f\in A\).
\end{lemma}

\begin{proof}
We write
\[
B_c:=\big(c^2L_{\mathcal G}^{\mathrm K}+m^2c^4\big)^{1/2}.
\]
Since
\[
\mathscr D_0^2
=
\mathcal L_{\mathcal G}^{\mathrm K}
\oplus
\mathcal L_{\mathcal G}^{\mathrm D},
\]
 for every Borel function \(\Psi\),
\[
\Psi(\mathscr D_0^2)\iota(f)
=
\iota\big(\Psi(L)f\big),
\]
whenever the two sides are well-defined. Since
\[
|\mathscr D_c|
=
\big(c^2\mathscr D_0^2+m^2c^4\big)^{1/2},
\]
we have
\[
|\mathscr D_c|^{-1}\iota(f)
=
\iota(B_c^{-1}f).
\]
Moreover,
\[
P_c^+
=
\frac12\left(I+\mathscr D_c|\mathscr D_c|^{-1}\right).
\]
Thus
\begin{equation}\label{eq:36}
z_c(f)
=
P_c^+\iota(f)
=
\frac12
\binom{
f+mc^2B_c^{-1}f
}{
-ic\,(B_c^{-1}f)'
}.
\end{equation}

Let \(E (\lambda)\) be the spectral resolution of \(L_{\mathcal G}^{\mathrm K}\), and define
\[
d\mu_f(\lambda)
:=
d\|E (\lambda)f\|_{L^2(\mathcal G)}^2.
\]
Since \(f\in H^1_{\mathrm K}(\mathcal G)\), one has
\[
\int_{[0,\infty)}\lambda\,d\mu_f(\lambda)
=
\|f'\|_{L^2(\mathcal G)}^2,
\]
moreover, there holds
\[
(B_cf,f)_{L^2}
=
\int_{[0,\infty)}
\big(c^2\lambda+m^2c^4\big)^{1/2}\,d\mu_f(\lambda),
\]
and
\[
\big(c^2\lambda+m^2c^4\big)^{1/2}-mc^2
=
\frac{\lambda}{\sqrt{m^2+\lambda/c^2}+m}
\longrightarrow
\frac{\lambda}{2m},
\]
\[
0
\le
\frac{\lambda}{\sqrt{m^2+\lambda/c^2}+m}
\le
\frac{\lambda}{2m},
\]
the dominated convergence theorem gives
\begin{equation}\label{eq:37}
(B_cf,f)_{L^2}
=
mc^2
+
\frac{1}{2m}\int_{\mathcal G}|f'|^2\,dx
+
o_c(1).
\end{equation}
Similarly,
\[
mc^2(B_c^{-1}f,f)_{L^2}
=
\int_{[0,\infty)}
\frac{m}{\sqrt{m^2+\lambda/c^2}}\,d\mu_f(\lambda).
\]
Since
\[
c^2
\left(
1-\frac{m}{\sqrt{m^2+\lambda/c^2}}
\right)
=
\frac{\lambda}{
\sqrt{m^2+\lambda/c^2}\,
\big(\sqrt{m^2+\lambda/c^2}+m\big)
}
\longrightarrow
\frac{\lambda}{2m^2},
\]
we obtain
\begin{equation}\label{eq:38}
mc^2(B_c^{-1}f,f)_{L^2}
=
1
-
\frac{1}{2m^2c^2}
\int_{\mathcal G}|f'|^2\,dx
+
o_c(c^{-2}).
\end{equation}

Using \(z_c(f)=P_c^+\iota(f)\) and the fact that \(P_c^+\) is an orthogonal projection, we get
\[
\|z_c(f)\|_{L^2}^2
=
\big(P_c^+\iota(f),\iota(f)\big)_{L^2}
=
\frac12
\left(
1+mc^2(B_c^{-1}f,f)_{L^2}
\right).
\]
Hence, by \eqref{eq:38},
\begin{equation}\label{eq:39}
\|z_c(f)\|_{L^2}^2
=
1
-
\frac{1}{4m^2c^2}
\int_{\mathcal G}|f'|^2\,dx
+
o_c(c^{-2}).
\end{equation}
Next, using again the functional calculus,
\[
\begin{aligned}
\|z_c(f)\|_c^2
&=
\big(|\mathscr D_c|z_c(f),z_c(f)\big)_{L^2}  \\
&=
\big(|\mathscr D_c|P_c^+\iota(f),\iota(f)\big)_{L^2} \\
&=
\frac12
\big(
(|\mathscr D_c|\iota(f),\iota(f))_{L^2}
+
(\mathscr D_c\iota(f),\iota(f))_{L^2}
\big).
\end{aligned}
\]
On the upper component,
\[
|\mathscr D_c|\iota(f)=\iota(B_cf),
\]
while
\[
(\mathscr D_c\iota(f),\iota(f))_{L^2}
=
mc^2\|f\|_{L^2}^2
=
mc^2.
\]
Therefore
\[
\|z_c(f)\|_c^2
=
\frac12
\left(
(B_cf,f)_{L^2}+mc^2
\right).
\]
By \eqref{eq:37},
\begin{equation}\label{eq:310}
\|z_c(f)\|_c^2
=
mc^2
+
\frac{1}{4m}
\int_{\mathcal G}|f'|^2\,dx
+
o_c(1).
\end{equation}
Combining \eqref{eq:39} and \eqref{eq:310}, we obtain
\begin{equation}\label{eq:311}
\|\Theta_c(f)\|_c^2
=
\frac{\|z_c(f)\|_c^2}{\|z_c(f)\|_{L^2}^2}
=
mc^2
+
\frac{1}{2m}
\int_{\mathcal G}|f'|^2\,dx
+
o_c(1).
\end{equation}

We also need the convergence of the nonlinear term. From \eqref{eq:36}, the same spectral estimates imply
\[
z_c(f)\longrightarrow \iota(f)
\quad
\text{in }H^1(\mathcal G,\mathbb C^2),
\]
uniformly for \(f\in A\). Indeed, the upper component is controlled by
\[
\int_{[0,\infty)}
(1+\lambda)
\left|
\frac{m}{\sqrt{m^2+\lambda/c^2}}-1
\right|^2
\,d\mu_f(\lambda),
\]
which tends to zero uniformly on \(A\). For the lower component one uses
\[
\frac{c^2}{4}
\int_{[0,\infty)}
\frac{\lambda}{c^2\lambda+m^2c^4}\,d\mu_f(\lambda)
=
\frac14
\int_{[0,\infty)}
\frac{\lambda}{\lambda+m^2c^2}\,d\mu_f(\lambda),
\]
and
\[
\frac{c^2}{4}
\int_{[0,\infty)}
\frac{\lambda^2}{c^2\lambda+m^2c^4}\,d\mu_f(\lambda)
=
\frac14
\int_{[0,\infty)}
\frac{\lambda^2}{\lambda+m^2c^2}\,d\mu_f(\lambda),
\]
which also tend to zero uniformly for \(f\in A\). Since \(\|z_c(f)\|_{L^2}\to1\) uniformly, we get
\[
\Theta_c(f)\longrightarrow \iota(f)
\quad
\text{in }H^1(\mathcal G,\mathbb C^2)
\]
uniformly for \(f\in A\). Hence, by the Sobolev embedding on the compact core,
\begin{equation}\label{eq:312}
\int_{\mathcal K}|\Theta_c(f)|^p\,dx
=
\int_{\mathcal K}|f|^p\,dx
+
o_c(1)
\end{equation}
uniformly for \(f\in A\).

Now let
\[
u_c=\Phi_c(\Theta_c(f))
=
t_c\Theta_c(f)+w_c,
\qquad
t_c\ge0,\quad w_c\in Y_c^-.
\]
By Proposition \(3.1\), applied with \(v=\Theta_c(f)\),
\[
\mathcal A(u_c)+\|w_c\|_c^2+\|w_c\|_{L^2}^2\|\Theta_c(f)\|_c^2
\le
\mathcal A(\Theta_c(f)),
\]
By \eqref{eq:312}, \(A(\Theta_c(f))\) is bounded uniformly for \(f\in A\). Therefore
\[
\|w_c\|_c\le C,
\qquad
\|w_c\|_{L^2}^2\le \frac{C}{mc^2}=o_c(1),
\]
and hence
\[
t_c^2=1-\|w_c\|_{L^2}^2=1+o_c(1).
\]
Moreover, the fractional Gagliardo--Nirenberg estimate on \(Y_c\) yields
\[
\|w_c\|_{L^p(\mathcal K)}^p
\le
\|w_c\|_{L^p(\mathcal G)}^p
\le
C c^{-\frac{p-2}{2}}\|w_c\|_c^{p-2}\|w_c\|_{L^2}^2
=
o_c(1).
\]
Consequently,
\[
u_c=t_c\Theta_c(f)+w_c
\longrightarrow
\iota(f)
\quad
\text{in }L^p(\mathcal K,\mathbb C^2),
\]
uniformly for \(f\in A\). Therefore
\begin{equation}\label{eq:313}
\int_{\mathcal K}|u_c|^p\,dx
=
\int_{\mathcal K}|f|^p\,dx
+
o_c(1).
\end{equation}

Finally, for \(u_c=\Phi_c(\Theta_c(f))=t_c\Theta_c(f)+w_c\),
we have
\[
\begin{aligned}
\mathcal J_c(\Theta_c(f))
&=
t_c^2\|\Theta_c(f)\|_c^2
-
\|w_c\|_c^2
-
\frac{2a}{p}\int_{\mathcal K}|u_c|^p\,dx  \\
&\le
\|\Theta_c(f)\|_c^2
-
\frac{2a}{p}\int_{\mathcal K}|u_c|^p\,dx.
\end{aligned}
\]
By \eqref{eq:311} and \eqref{eq:313},
\[
\mathcal J_c(\Theta_c(f))
\le
mc^2
+
\frac{1}{2m}\int_{\mathcal G}|f'|^2\,dx
-
\frac{2a}{p}\int_{\mathcal K}|f|^p\,dx
+
o_c(1).
\]
That is,
\[
\mathcal J_c(\Theta_c(f))
\le
mc^2+\I_{\infty}(f)+o_c(1),
\]
uniformly for \(f\in A\). The proof is complete.
\end{proof}

\begin{lemma}\label{lem:PSbelow}
For $c$ sufficiently large, the reduced functional $\mathcal J_c$ satisfies the Palais--Smale condition at every level
\begin{equation}\label{eq:PSbelowThreshold}
d<mc^2.
\end{equation}

\end{lemma}

\begin{proof}
Let $v_n\subset S_c^+$ be such that
\[
\mathcal J_c(v_n)\to d<mc^2,
\qquad
\|d\mathcal J_c(v_n)\|\to0.
\]
Set $u_n=\Phi_c(v_n)=t_nv_n+w_n$. Since $v_n\in\mathcal M_c(v_n)$,
\begin{equation}\label{eq:coercivityReduced}
\mathcal J_c(v_n)\ge\mathcal I_c(v_n)
=\|v_n\|_c^2-\frac{2a}{p}\int_\K |v_n|^p\dd x.
\end{equation}
By Corollary \ref{lem:GNDirac} 
\[
\mathcal I_c(v_n)
\ge \|v_n\|_c^2-Cc^{-\frac{p-2}{2}}\|v_n\|_c^{p-2},
\]
which yields $(v_n)$ is bounded in $Y_c$. 
Proposition \ref{prop:reduction} shows the boundedness of $(w_n)$ in $Y_c^-$, hence $(u_n)$ is bounded in $Y_c$.
Moreover, one readily checks that
\begin{equation}\label{eq:omegaFormula}
\omega_n:=\omega(v_n)=(u_n^+,u_n^+)_c-(u_n^-,u_n^-)_c-a\int_\K |u_n|^p\dd x,
\end{equation}
and
\begin{equation}\label{eq:energyOmegaRelation}
\mathcal I_c(u_n)-\omega_n
=\frac{a(p-2)}{p}\int_\K |u_n|^p\dd x,
\end{equation}
which yields
\begin{equation}\label{eq:omegaUpper}
\limsup_{n\to\infty}\omega_n\le d<mc^2.
\end{equation}
It follows from Proposition \ref{prop:reduction}, up to a subsequence, $\omega_n\to\omega\in[0,mc^2)$, 
and $(u_n)$ is a bounded Palais-Smale sequence of $\I_c^\omega$.
We may assume $u_n\rightharpoonup u$ in $Y_c$ and $u_n\to u$ in $L^p(\K,\C^2)$, then
$$ \begin{aligned} &\omega\left\|u_{n}^{+}-u^{+}\right\|_{L^{2}}^{2}
   =\frac{1}{2}\left(\dd \mathcal{I}_{c} \left(u_{n}\right)
   - \dd \mathcal{I}_{c}(u)\right) \left[u_{n}^{+}-u^{+}
   \right]+o_{n}(1) \\ &=\left\|u_{n}^{+}-u^{+}\right\|_{c}^{2}
    -\Re \int_{\K} \left(\left|u_{n}\right|^{p-2}
     u_{n}-|u|^{p-2} u\right) \\ &\quad 
     \cdot\left(u_{n}^{+}-u^{+}\right) d x+o_{n}(1) 
     \\ &=\left\|u_{n}^{+}-u^{+}\right\|_{c}^{2}+o_{n}(1).
     \end{aligned} $$
Combining this with the fact that $\omega<mc^2$, we get
\[
\left\|u_{n}^{+}-u^{+}\right\|_{c}=o_n(1),
\]
similarly,
\[
\left\|u_{n}^{-}-u^{-}\right\|_{c}=o_n(1),
\]
which implies that
\[
\|u_n-u\|_c=o_n(1)
\]
and so the proof of this lemma is completed.

\end{proof}

\section{The Mass-Supercritical Case}\label{sec:supercritical}

Throughout this section we assume
\begin{equation}\label{eq:superRange}
4<p< 6.
\end{equation}
Choose once and for all a number $s$ such that
\begin{equation}\label{eq:sChoiceGraph}
1<s<\frac{p-2}{2(p-4)}.
\end{equation}
For $u\in Y_c$ set
\begin{equation}\label{eq:TcExcess}
T_c(u):=\|u\|_c^2-mc^2\|u\|_{L^2}^2\ge0.
\end{equation}
The local set on which the reduced functional will be used is
\begin{equation}\label{eq:OcGraph}
\mathcal O_c:=\{u\in Y_c:\|u\|_{L^2}\le1,\ \|u\|_c<c^s\},
\qquad
\mathcal O_c^+:=\mathcal O_c\cap S_c^+.
\end{equation}
Since $s>1$, every nonrelativistic test state constructed in Lemma \ref{lem:FWgraph} belongs to $\mathcal O_c^+$ for all large $c$.

{

\begin{lemma}\label{lem:thresholdFinite}
Let \(\mathcal G\) be a connected noncompact metric graph with finitely
many edges and vertices, and let \(\K\) be its compact core. For \(c>0\), set
\[
\mathcal{N}_c:=\ker(|\mathscr D_c|-mc^2)
=
\ker(\mathscr D_c-mc^2)\oplus\ker(\mathscr D_c+mc^2).
\]
Then
\[
\dim_{\mathbb C}\ker(\mathscr D_c-mc^2)=0
\]
and
\[
\dim_{\mathbb C}\ker(\mathscr D_c+mc^2)=b_1(K),
\]
where \(b_1(K)\) is the first Betti number of the compact core. Hence
\[
\dim_{\mathbb C}\mathcal{N}_c=b_1(K),
\]
which is independent of \(c\). In particular, \(\dim_{\mathbb C}N_c\)
is uniformly bounded for \(c>1\).

Moreover, for every \(2\le q<\infty\), there exists a constant
\(C_q>0\), depending only on \(q\) and on the metric graph, but not on
\(c\), such that
\begin{equation}\label{as}
       \|\Pi_c u\|_{L^q(K,\mathbb C^2)}^q
\le
C_q\|\Pi_c u\|_{L^2(\mathcal G,\mathbb C^2)}^q,
\qquad
u\in L^2(\mathcal G,\mathbb C^2), 
\end{equation}
where \(\Pi_c\) is the \(L^2\)-orthogonal projection onto \(\mathcal{N}_c\).
\end{lemma}

\begin{proof}
Let $E_\K$ be the set of all bounded edges, and $V_\K$ the set of all vertices in $\K$.
 \(e\in E_\K\) is identified
with an interval \([0,\ell_e]\) and is given an arbitrary orientation.
We denote its initial and terminal vertices by \(e_-\) and \(e_+\),
respectively. If \(e\succ v\) and \(e\) is not a
loop, \(v_e\) denotes the other endpoint of \(e\). If \(e\) is a loop
based at \(v\), we set \(v_e=v\). Finally, let
\[
V_\infty:=
\{v\in V_\K:\ v \text{ is incident with at least one half-line}\}.
\]
Since \(\mathcal G\) is connected and noncompact, every connected
component of \(\K\) meets \(V_\infty\).

We begin with \(\ker(\mathscr D_c-mc^2)\). Let
\[
\psi=(\psi^1,\psi^2)^T\in\ker(\mathscr D_c-mc^2).
\]
On each edge,
\[
-ic(\psi^2)' + mc^2\psi^1 = mc^2\psi^1,
\qquad
-ic(\psi^1)' - mc^2\psi^2 = mc^2\psi^2.
\]
Hence
\[
(\psi^2)'=0,
\qquad
(\psi^1)'=2imc\,\psi^2.
\]
Therefore, on every bounded edge \(e\in E_\K\),
\[
\psi_e^2(x)=A_e,
\qquad
\psi_e^1(x)=2imc\,A_ex+B_e,
\]
for some \(A_e,B_e\in\mathbb C\). On every half-line, \(\psi\in L^2(\mathcal G,\mathbb C^2)\) force
\(A_e=B_e=0\).

The continuity condition for the first spinorial component gives a
well-defined complex number
\[
F_v:=\psi^1(v),\qquad v\in V_\K.
\]
If \(e\in E_\K\) is oriented from \(e_-\) to \(e_+\), then
\[
F_{e_-}=B_e,
\qquad
F_{e_+}=2imc\,A_e\ell_e+B_e.
\]
Consequently,
\[
A_e=\frac{F_{e_+}-F_{e_-}}{2imc\,\ell_e}.
\]
Moreover, if \(v\in V_\infty\), 
\(
F_v=0.
\)
We now rewrite the balance condition for the second component. At a
vertex \(v\in V_\K\), the Kirchhoff-type condition reads
\[
\sum_{e\succ v}\psi_e^2(v)^\pm=0.
\]
With the above orientation convention,
\[
\psi_e^2(v)^\pm
=
\begin{cases}
 A_e, & v=e_-,\\
-A_e, & v=e_+.
\end{cases}
\]
Using the expression for \(A_e\), one obtains in both cases
\[
\psi_e^2(v)^\pm
=
\frac{F_{v_e}-F_v}{2imc\,\ell_e}.
\]
Indeed, if \(v=e_-\), then \(v_e=e_+\) and
\[
\psi_e^2(v)^\pm=A_e
=
\frac{F_{e_+}-F_v}{2imc\,\ell_e}
=
\frac{F_{v_e}-F_v}{2imc\,\ell_e};
\]
if \(v=e_+\), then \(v_e=e_-\) and
\[
\psi_e^2(v)^\pm=-A_e
=
-\frac{F_v-F_{e_-}}{2imc\,\ell_e}
=
\frac{F_{v_e}-F_v}{2imc\,\ell_e}.
\]
Hence the Kirchhoff condition yields
\[
\sum_{e\succ v}\frac{F_v-F_{v_e}}{\ell_e}=0,
\qquad v\in V_\K.
\]
Multiplying this identity by \(\overline{F_v}\), summing over
\(v\in V_\K\), and using the usual discrete integration by parts, we get
\[
0
=
\sum_{v\in V_\K}
\sum_{e\succ v}
\frac{F_v-F_{v_e}}{\ell_e}\overline{F_v}
=
\sum_{e\in E_\K}
\frac{|F_{e_+}-F_{e_-}|^2}{\ell_e}.
\]
Therefore \(F\) is constant on each connected component of \(K\). Since
each connected component of \(K\) meets \(V_\infty\), and \(F_v=0\) on
\(V_\infty\), we obtain
\[
F_v=0,\qquad v\in V_\K.
\]
It follows that \(B_e=0\) and \(A_e=0\) for every bounded edge \(e\).
Thus
\[
\ker(\mathscr D_c-mc^2)=\{0\}.
\]

We next compute \(\ker(\mathscr D_c+mc^2)\). Let
\[
\psi=(\psi^1,\psi^2)^T\in\ker(\mathscr D_c+mc^2).
\]
On each edge,
\[
-ic(\psi^2)' + mc^2\psi^1 = -mc^2\psi^1,
\qquad
-ic(\psi^1)' - mc^2\psi^2 = -mc^2\psi^2.
\]
Hence
\[
(\psi^1)'=0,
\qquad
(\psi^2)'=-2imc\,\psi^1.
\]
Thus, on every bounded edge \(e\in E_\K\),
\[
\psi_e^1(x)=E_e,
\qquad
\psi_e^2(x)=-2imc\,E_ex+F_e,
\]
for some \(E_e,F_e\in\mathbb C\). On every half-line,
\(L^2\)-integrability again forces both constants to vanish.

The continuity condition for the first component gives numbers
\(G_v\in\mathbb C\), \(v\in V_\K\), such that
\[
E_e=G_{e_-}=G_{e_+}.
\]
Moreover, if \(v\in V_\infty\), \(G_v=0\). Therefore \(G\) is constant on each
connected component of \(\K\) and vanishes on that component. Hence
\[
E_e=0,\qquad e\in E_\K.
\]
Consequently, every element of \(\ker(\mathscr D_c+mc^2)\) has the form
\[
\psi_e(x)=
\begin{pmatrix}
0\\
F_e
\end{pmatrix}
\quad\text{on each bounded edge }e\in E_\K,
\qquad
\psi_e\equiv0
\quad\text{on each half-line}.
\]

The only remaining condition is the balance condition for the second
component:
\[
\sum_{e\succ v}\psi_e^2(v)^\pm=0,
\qquad v\in V_\K.
\]
Equivalently,
\[
\sum_{e:e_-=v}F_e-\sum_{e:e_+=v}F_e=0,
\qquad v\in V_\K.
\]
Let \(B\) be the oriented incidence matrix of the compact graph \(\K\),
defined by
\[
B_{ve}
=
\begin{cases}
1, & v=e_-,\\
-1, & v=e_+,\\
0, & \text{otherwise},
\end{cases}
\]
with the two contributions cancelling if \(e\) is a loop. Then the
preceding condition is exactly
\[
BF=0,
\qquad
F=(F_e)_{e\in E_\K}.
\]
Therefore
\[
\ker(\mathscr D_c+mc^2)\simeq\ker B.
\]

Let \(\kappa(\K)\) be the number of connected components of \(\K\). The
rank of the incidence matrix of a finite graph is
\[
\operatorname{rank}B=|V_\K|-\kappa(\K).
\]
Indeed, on each connected component the sum of the rows of the incidence
matrix is zero, so the rank is at most the number of vertices in that
component minus one. Conversely, the columns corresponding to any
spanning tree of that component have rank exactly the number of vertices
minus one. Summing over all connected components gives the formula.
Hence
\[
\dim_{\mathbb C}\ker B
=
|E_\K|-|V_\K|+\kappa(K)
=
b_1(K).
\]
Thus
\[
\dim_{\mathbb C}\ker(\mathscr D_c+mc^2)=b_1(K).
\]
Combining this with \(\ker(\mathscr D_c-mc^2)=\{0\}\), we obtain
\[
\dim_{\mathbb C}\mathcal{N}_c=b_1(K).
\]
In particular, this dimension is independent of \(c\).
Since all norms on a finite dimensional space are equivalent, this proves \eqref{as}.
\end{proof}

}

\begin{lemma}\label{lem:thresholdMGN}
Let
\[
 \Lambda_c:=|\mathscr{D}_c|-mc^2,
\qquad
T_c(u):=\| \Lambda_c^{1/2}u\|_{L^2}^2=\|u\|_c^2-mc^2\|u\|_{L^2}^2,
\]
and let $\Pi_c $ be the threshold projection of Lemma \ref{lem:thresholdFinite}. Then there exist constants $C>0$ and $c_0>0$, independent of $c\ge c_0$, such that every $u\in Y_c$ satisfies
\begin{equation}\label{eq:modifiedGraphGN}
\int_\G |u|^p\dd x
\le C\left(
\|\Pi_c u\|_{L^2}^p
+\|u_c^\perp\|_{L^2}^{\frac{p+2}{2}}T_c(u)^\theta
+c^{-\frac{p-2}{2}}\|u_c^\perp\|_{L^2}^2T_c(u)^{2\theta}
\right),
\end{equation}
where $\theta=(p-2)/4$ and
\[
 u_c^\perp:=(I-\Pi_c )u.
\]

\end{lemma}

\begin{proof}
We give a detailed proof, because this is the point where the Euclidean Fourier argument has to be replaced by spectral calculus on the graph. 
Let $E_c(\cdot)$ be the spectral resolution of the nonnegative self-adjoint operator
\[
 \Lambda_c=|\mathscr{D}_c|-mc^2.
\]
Put
\[
\nu_c:=u_c^\perp=(I-\Pi_c )u,
\qquad
\nu_\ell:=E_c((0,c^2])\nu_c,
\qquad
\nu_h:=E_c((c^2,+\infty))\nu_c.
\]
Then
\[
\nu_c=\nu_\ell+\nu_h,
\qquad
\nu_\ell\perp \nu_h\quad\text{in }L^2,
\]
and, by the spectral theorem,
\begin{equation}\label{eq:spectralMeasureTcDetailed}
T_c(u)=T_c(\nu_c)=T_c(\nu_\ell)+T_c(\nu_h),
\qquad
T_c(w)=\int_{[0,+\infty)} \lambda\,\dd\mu_w^c( \lambda),
\end{equation}
where
\[
\mu_w^c(B):=\|E_c(B)w\|_{L^2}^2
\]
is the spectral measure associated with $ \Lambda_c$ and $w$.  Notice that the endpoint $0$ has been removed from the definition of $\nu_\ell$, because the corresponding spectral subspace is precisely $\mathcal N_c $ and has already been projected out.

We now turn to the estimation of the low spectral part $\nu_\ell$. For
$w\in\operatorname{dom}(\D_c)$,
It follows from the self-adjointness of the Dirac operator that
\begin{equation}
        \|\D_cw\|_{L^2}^2=m^2c^4\|w\|_{L^2}^2+c^2\|w'\|_{L^2}^2 .
\label{eq:5}
\end{equation}
Since $v_\ell$ has spectral support in $\{0< \lambda\le c^2\}$ for $\Lambda_c$, we have
$v_\ell\in\operatorname{dom}(\D_c)$ and, using $|\D_c|=mc^2+\Lambda_c$ on this support,
\begin{equation}
  \begin{split}
     c^2\|v_\ell'\|_{L^2}^2
 &=\|\D_cv_\ell\|_{L^2}^2-m^2c^4\|v_\ell\|_{L^2}^2  \\
 &=\int_{(0,c^2]}\bigl((mc^2+ \lambda)^2-m^2c^4\bigr)\,d\mu_{v_\ell}( \lambda) \\ 
 &=\int_{(0,c^2]}(2mc^2 \lambda+ \lambda^2)\,d\mu_{v_\ell}( \lambda)  \\
 &\le (2m+1)c^2\int_{(0,c^2]} \lambda\,d\mu_{v_\ell}( \lambda)\\
  &=(2m+1)c^2T_c(v_\ell),
  \end{split}
\label{eq:6}
\end{equation}
for $c\ge1$.  Hence
\begin{equation}
        \|v_\ell'\|_{L^2}^2\le C T_c(v_\ell).
\label{eq:7}
\end{equation}
Applying Lemma \ref{lem:graph_Hhalf_GN} to $v_\ell$ and using \eqref{eq:7} and \eqref{eq:spectralMeasureTcDetailed}, we obtain
\begin{equation}
\begin{aligned}
        \|v_\ell\|_{L^p(\mathcal G)}^p
        &\le C\|v_\ell\|_{L^2}^{\frac{p+2}{2}}
                 \|v_\ell'\|_{L^2}^{\frac{p-2}{2}}  \\
        &\le C\|v\|_{L^2}^{\frac{p+2}{2}}T_c(v_\ell)^{\frac{p-2}{4}}  \\
        &\le C\|u_c^\perp\|_{L^2}^{\frac{p+2}{2}}T_c(u)^\theta .
\end{aligned}
\label{eq:9}
\end{equation}

It remains to estimate the high spectral part.  We first recall that
\begin{equation}
        \|w\|_{H^{1/2}(\mathcal G)}^2\le Cc^{-1}\|w\|_c^2,
        \qquad w\in Y_c,
\label{eq:10}
\end{equation}
valid for all sufficiently large $c$. 
\begin{comment}

To see this, let $D_0=-i\sigma_1\frac{d}{dx}$ with the
same vertex conditions.  Since
\[
        D_c^2=c^2D_0^2+m^2c^4I,
        \qquad
        |D_c|=c(D_0^2+m^2c^2I)^{1/2},
\]
the spectral theorem gives, for $c\ge c_0$ with $m^2c^2\ge1$,
\[
\begin{aligned}
        \|w\|_{H^{1/2}(\mathcal G)}^2
        &\le C\|(I+D_0^2)^{1/4}w\|_{L^2}^2  \\
        &= C\langle (I+D_0^2)^{1/2}w,w\rangle  \\
        &\le C\langle (D_0^2+m^2c^2I)^{1/2}w,w\rangle
        =Cc^{-1}\|w\|_c^2.
\end{aligned}
\]
This proves \eqref{eq:10}.
\end{comment}
On the support of $v_h$ one has $\Lambda_c>c^2$. Therefore
\begin{equation}
        \|v_h\|_{L^2}^2\le c^{-2}T_c(v_h),
\label{eq:11}
\end{equation}
and hence
\begin{equation}
        \|v_h\|_c^2
        =mc^2\|v_h\|_{L^2}^2+T_c(v_h)
        \le (m+1)T_c(v_h).
\label{eq:12}
\end{equation}
Combining \eqref{eq:10} and \eqref{eq:12} yields
\begin{equation}
        \|v_h\|_{H^{1/2}(\mathcal G)}^2\le Cc^{-1}T_c(v_h).
\label{eq:13}
\end{equation}
Thus, by Lemma \ref{lem:graph_Hhalf_GN}, together with \eqref{eq:13}, \eqref{eq:spectralMeasureTcDetailed}, and $\|v_h\|_{L^2}\le\|v\|_{L^2}$,
\begin{equation}
\begin{aligned}
        \|v_h\|_{L^p(\mathcal G)}^p
        &\le C\|v_h\|_{L^2}^2\|v_h\|_{H^{1/2}(\mathcal G)}^{p-2}  \\
        &\le Cc^{-\frac{p-2}{2}}\|v\|_{L^2}^2T_c(v_h)^{\frac{p-2}{2}} \\
        &\le Cc^{-\frac{p-2}{2}}\|u_c^\perp\|_{L^2}^2T_c(u)^{2\theta}.
\end{aligned}
\label{eq:15}
\end{equation}
Finally, combining \eqref{as}, \eqref{eq:9}, and \eqref{eq:15} yields
\[
\begin{aligned}
\int_{\mathcal G}|u|^p\,dx
&\le C\left(
        \|z\|_{L^p(\mathcal G)}^p
        +\|v_\ell\|_{L^p(\mathcal G)}^p
        +\|v_h\|_{L^p(\mathcal G)}^p
      \right) \\
&\le C\left(
   \|\Pi_cu\|_{L^2}^p
   +\|u_c^\perp\|_{L^2}^{\frac{p+2}{2}}T_c(u)^\theta
   +c^{-\frac{p-2}{2}}\|u_c^\perp\|_{L^2}^2T_c(u)^{2\theta}
\right).
\end{aligned}
\]
\end{proof}

\begin{lemma}\label{lem:graphPseudoGN}
Let $4<p<6$, set $\theta=(p-2)/4$, and choose $s$ as in \eqref{eq:sChoiceGraph}. Then there exist constants $C>0$, $\eta>0$, and $c_0>0$ such that the following assertions hold for every $c\ge c_0$.

\begin{enumerate}[label=\textup{(\roman*)}]
\item If $u\in\mathcal O_c$, then, with $u_c^\perp=(I-\Pi_c )u$,
\begin{equation}\label{eq:localModifiedGraphGN}
\int_\G |u|^p\dd x
\le C\left(
\|\Pi_c u\|_{L^2}^p
+\|u_c^\perp\|_{L^2}^{\frac{p+2}{2}}T_c(u)^\theta
+c^{-\eta}\|u_c^\perp\|_{L^2}^2T_c(u)
\right).
\end{equation}

\item If $u\in Y_c$ and $\|u\|_c\le c^s$, then
\begin{equation}\label{eq:localNonlinearSmall}
\int_\G |u|^p\dd x
\le C c^{-\eta}\|u\|_{L^2}^2\|u\|_c^2.
\end{equation}

\item If $u\in Y_c$, $\|u\|_c\le  c^s$, and $h\in Y_c$, then there is a number
\(
\gamma\in(0,2]
\),
depending only on $p$, such that
\begin{equation}\label{eq:localHessianEstimateGeneral}
\int_\G |u|^{p-2}|h|^2\dd x
\le C c^{-\eta}\|u\|_{L^2}^\gamma\|h\|_c^2.
\end{equation}
\item If $u\in Y_c$, $\|u\|_c\le  c^s$, $\|u\|_{L^2}\le 1$, and $h\in Y_c$, there holds
\begin{equation}\label{eq}
\int_\G |u|^{p-1}|h|\dd x
\le C c^{-\eta}\|u\|_c\|h\|_c.
\end{equation}
\end{enumerate}
\end{lemma}

\begin{proof}
        \begin{enumerate}[label=\textup{(\roman*)}]
                \item
We first derive \eqref{eq:localModifiedGraphGN} from Lemma \ref{lem:thresholdMGN}. If $u\in\mathcal O_c$, then $T_c(u)\le \|u\|_c^2\le c^{2s}$. Since $2\theta-1=(p-4)/2>0$,
\[
c^{-\frac{p-2}{2}}T_c(u)^{2\theta}
=c^{-\frac{p-2}{2}}T_c(u)^{\frac{p-4}{2}}T_c(u)
\le c^{-\frac{p-2}{2}+s(p-4)}T_c(u).
\]
The choice \eqref{eq:sChoiceGraph} makes the exponent negative. Choosing
\[
0<\eta<\frac{p-2}{2}-s(p-4)
\]
and using $\|u_c^\perp\|_{L^2}\le\|u\|_{L^2}\le1$ on $\mathcal O_c$ gives \eqref{eq:localModifiedGraphGN}.

\item Corollary \ref{lem:GNDirac} gives, for $u\in Y_c$,
\[
\int_\G |u|^p\dd x
\le C c^{-\frac{p-2}{2}}\|u\|_c^{p-2}\|u\|_{L^2}^2.
\]
If $\|u\|_c\le c^s$, then
\[
\int_\G |u|^p\dd x
\le C c^{-\frac{p-2}{2}+s(p-4)}\|u\|_{L^2}^2\|u\|_c^2.
\]
After decreasing $\eta$ if necessary, \eqref{eq:sChoiceGraph} gives \eqref{eq:localNonlinearSmall}.

\item Choose $q>p-2$ close enough to $p-2$ and let $r\in(2,+\infty)$ be such that
\begin{equation}\label{eq:qrLocalRelation}
\frac{p-2}{q}+\frac2r=1.
\end{equation}
Using Corollary \ref{lem:GNDirac} and $\|u\|_c\le c^s$, we get
\[
\|u\|_{L^q(\G)}^{p-2}
\le C c^{\left(1-\frac2q\right)(p-2)(s-\frac12)}
\|u\|_{L^2}^{\frac{2(p-2)}q}.
\]
Similarly,
\[
\|h\|_{L^r(\G)}^2
\le C c^{-1-2/r}\|h\|_c^2.
\]
H\"older's inequality yields
\[
\int_\G |u|^{p-2}|h|^2\dd x
\le C_R c^{E(q)}\|u\|_{L^2}^{\frac{2(p-2)}q}\|h\|_c^2,
\]
where
\[
E(q)=\left(1-\frac2q\right)(p-2)\left(s-\frac12\right)-1-\frac2r.
\]
Using \eqref{eq:qrLocalRelation} and letting $q\to p-2$ gives
\[
E(q)\longrightarrow (p-4)s-\frac{p-2}{2}<0
\]
by \eqref{eq:sChoiceGraph}. Taking $q$ sufficiently close to $p-2$, setting $\gamma=2(p-2)/q$, and decreasing $\eta$ if necessary proves \eqref{eq:localHessianEstimateGeneral}.
\item
By Cauchy inequality, \eqref{eq:localNonlinearSmall} and \eqref{eq:localHessianEstimateGeneral}, we get
\[
        \int_\G |u|^{p-1}|h|\,dx
        \le C\Bigl(\int_\G |u|^p\,dx\Bigr)^{1/2}
              \Bigl(\int_\G |u|^{p-2}|h|^2\,dx\Bigr)^{1/2}
        \le Cc^{-\eta}\|u\|_c\|h\|_c .
\]
\end{enumerate}
\end{proof}

{

\begin{lemma}\label{lem:local-fibre-PS-supercritical}
There exists $c_1>0$ such that the following holds for every $c\ge c_1$ and every
$v\in\mathcal O_c^+$.  Let $(u_n)\subset\mathcal M_c(v)$ be a Palais--Smale sequence for
$\mathcal I_c|_{\mathcal M_c(v)}$ at a positive level $d>0$, namely
\begin{equation}\label{110}
   \mathcal I_c(u_n)\to d>0,
        \qquad
        \|d(\mathcal I_c|_{\mathcal M_c(v)})(u_n)\|\to0 .
\end{equation}
Writing
\[
        u_n=t_nv+w_n,
        \qquad t_n\ge0,
        \qquad w_n\in Y_c^-,
        \qquad t_n^2+\|w_n\|_{L^2}^2=1,
\]
one has:
\begin{enumerate}[label=\textup{(\roman*)}]
\item $(u_n)$ is bounded in $Y_c$.
\item $\omega_n:=\omega(u_n)$ is bounded and
\begin{equation}\label{eq:local-omega-positive-supercritical}
        \liminf_{n\to\infty}\omega_n
        \ge (1-Cc^{-\eta})\|v\|_c^2>0 .
\end{equation}
\item After passing to a subsequence, $u_n\to u$ strongly in $Y_c$, for some
$u\in\mathcal M_c(v)$.
\end{enumerate}
Here and below $C>0$ is independent of $c$, $v$ and $n$.
\end{lemma}

\begin{proof}
        \begin{enumerate}[label=\textup{(\roman*)}]
                \item
For large $n$, $\mathcal I_c(u_n)\ge 0$.  Hence
\[
       0\le \|u_n^+\|_c^2-\|u_n^-\|_c^2-\mathcal A(u_n)
\]
Thus $u_n^-$  is bounded in $Y_c$, also, $u_n$  is bounded in $Y_c$.  
\item
It follows from \eqref{110}
that
\begin{equation}\label{eq:local-fibre-PS-Euler-supercritical}
       \sup_{\substack{\|h\|_c=1 \\ h\in Y_c^-\oplus\operatorname{span}\{v\}}} \bigl|\dd\mathcal I_c(u_n)[h]
        -2\omega_n\Re \int_G u_n\cdot\overline h\,dx\bigr|
        =o_n(1),
\end{equation}
Testing \eqref{eq:local-fibre-PS-Euler-supercritical} with $h=u_n^+$ gives
\begin{equation}\label{eq:positive-test-omega-supercritical}
        \omega_n \|u_n^+\|_{L^2}^2
        =\|u_n^+\|_c^2-\frac12\dd\mathcal A(u_n)[u_n^+]+o_(1).
\end{equation}
By \eqref{eq},
\begin{equation}\label{eq:Aprime-positive-supercritical}
        \frac12|\dd\mathcal A(u_n)[u_n^+]|
        \le Cc^{-\eta}\|u_n\|_c\|u_n^+\|_c
        \le Cc^{-\eta}\|u_n^+\|_c^2 .
\end{equation}
Combining \eqref{eq:positive-test-omega-supercritical} with
\eqref{eq:Aprime-positive-supercritical}, we obtain
\[
        \liminf_{n\to\infty}\omega_n
        \ge (1-Cc^{-\eta})\|v\|_c^2 .
\]
 \item By testing
\eqref{eq:local-fibre-PS-Euler-supercritical} with $u_n$ and using
$\dd\mathcal A(u_n)[u_n]=p\mathcal A(u_n)$, one obtains boundedness of $(\omega_n)$.  Hence, up to a subsequence, we can assume
\[
        \omega_n\to\omega>0,\quad u_n\rightharpoonup u\quad\hbox{weakly in }Y_c .
\]
Since the nonlinear term is supported on the compact core $\K$ and
$Y_c\hookrightarrow L^p(\K,\mathbb C^2)$ compactly, we have
\begin{equation}\label{eq:compact-core-PS-supercritical}
        u_n\to u
        \quad\hbox{strongly in }L^p(\K,\mathbb C^2).
\end{equation}
Taking $h=u_n^--u^-$ in \eqref{eq:local-fibre-PS-Euler-supercritical}, we get
\[
        -2(u_n^-,u_n^--u^-)_c
        -\dd\mathcal A(u_n)[u_n^--u^-]
        -2\omega_n\Re \int_\G u_n^-\cdot\overline{(u_n^--u^-)}\,dx
        \to0 .
\]
By \eqref{eq:compact-core-PS-supercritical},
$\dd\mathcal A(u_n)[u_n^--u^-]\to0$. Consequently,
\[
        \|u_n^--u^-\|_c^2+\omega\|u_n^--u^-\|_{L^2}^2\to0 .
\]
Thus $u_n^-\to u^-$ strongly in $Y_c$.
Since \(u_n^+=t_nv\), after passing to a subsequence \(t_n\to t\), hence \(u_n^+\to tv\) strongly in \(Y_c\). 
Together with \(u_n^-\to u^-\), this gives \(u_n\to u\) strongly in \(Y_c\).

\end{enumerate}
\end{proof}

\begin{lemma}\label{lem:local-fibre-concavity-supercritical}
Let $4<p<6$, $\eta$ as in Lemma~\ref{lem:graphPseudoGN}.
There exists $c_2>0$ such that, for every $c\ge c_2$ and every
$v\in\mathcal O_c^+$, the following holds.  For $w\in Y_c^-$ with $\|w\|_{L^2}<1$, define
\[
        a(w):=(1-\|w\|_{L^2}^2)^{1/2},
        \qquad
        \mathcal F_v(w):=\mathcal I_c(a(w)v+w).
\]
Then $\mathcal F_v$ is strictly concave on
\begin{equation}\label{eq:local-concavity-region-supercritical}
        \mathcal U_c(v):=\left\{w\in Y_c^-:\
        \|w\|_c^2+\|w\|_{L^2}^2\|v\|_c^2\le 2\mathcal A(v)\right\}.
\end{equation}
More precisely,
\begin{equation}\label{eq:local-concavity-estimate-supercritical}
       \dd^2 \mathcal F_v(w)[\xi,\xi]
        \le -\|\xi\|_c^2,
        \qquad w\in\mathcal U_c(v),\quad \xi\in Y_c^- .
\end{equation}
\end{lemma}

\begin{proof}
By \eqref{eq:localNonlinearSmall}, for $v\in\mathcal O_c^+$,
\begin{equation}\label{eq:A-v-local-small-supercritical}
        \mathcal A(v)
        \le Cc^{-\eta}\|v\|_c^2 .
\end{equation}
Thus, if $w\in\mathcal U_c(v)$, then
\begin{equation}\label{eq:w-small-local-supercritical}
        \|w\|_{L^2}^2\le Cc^{-\eta},
        \qquad
        \|w\|_c\le Cc^{-\eta/2}\|v\|_c .
\end{equation}
For $c$ sufficiently large, $\|w\|_{L^2}<1/2$ and $a(w)\ge1/2$.  If
$u=a(w)v+w$, then
\begin{equation}\label{eq:u-local-window-concavity-supercritical}
        \|u\|_{L^2}=1,
        \qquad
        \|u\|_c\le (1+Cc^{-\eta/2})\|v\|_c<2c^s .
\end{equation}
Therefore, by \eqref{eq},
\begin{equation}\label{eq:Aprime-v-local-supercritical}
        |\dd\mathcal A(u)[v]|
        \le Cc^{-\eta}\|u\|_c\|v\|_c
        \le Cc^{-\eta}\|v\|_c^2 .
\end{equation}

For $\xi\in Y_c^-$,
\[
       \dd a(w)[\xi]
        =-a(w)^{-1}\Re \int_G w\cdot\overline\xi\,dx,
\]
and
\[
        \dd^2a(w)[\xi,\xi]
        =-a(w)^{-1}\|\xi\|_{L^2}^2
        -a(w)^{-3}\left(\Re \int_G w\cdot\overline\xi\,dx\right)^2 .
\]
Consequently,
\begin{equation}\label{eq:a-identity-local-supercritical}
        (\dd a(w)[\xi])^2+a(w)\dd^2a(w)[\xi,\xi]
        =-\|\xi\|_{L^2}^2,
\end{equation}
and, using $a(w)\ge1/2$,
\begin{equation}\label{eq:a-second-bound-local-supercritical}
        |\dd^2a(w)[\xi,\xi]|\le C\|\xi\|_{L^2}^2 .
\end{equation}
Put
\[
        \zeta:=\dd a(w)[\xi]v+\xi .
\]
Using
\eqref{eq:a-identity-local-supercritical}, we get
\begin{equation}\label{eq:F-second-local-supercritical}
        \dd^2\mathcal F_v(w)[\xi,\xi]
        =-2\|\xi\|_c^2-2\|v\|_c^2\|\xi\|_{L^2}^2
        -\dd^2\mathcal A(u)[\zeta,\zeta]
        -\dd\mathcal A(u)\left[\dd^2a(w)[\xi,\xi]v\right] .
\end{equation}
Since $z\mapsto |z|^p$ is convex, $\dd^2\mathcal A(u)[\zeta,\zeta]\ge0$.  By
\eqref{eq:Aprime-v-local-supercritical} and \eqref{eq:a-second-bound-local-supercritical},
\[
        |\dd\mathcal A(u) \left[\dd^2a(w)[\xi,\xi]v\right]|
        \le Cc^{-\eta}\|v\|_c^2\|\xi\|_{L^2}^2 .
\]
Therefore
\[
        \dd^2\mathcal F_v(w)[\xi,\xi]
        \le -2\|\xi\|_c^2
        -\bigl(2-Cc^{-\eta}\bigr)\|v\|_c^2\|\xi\|_{L^2}^2 .
\]
Taking $c$ larger if necessary gives \eqref{eq:local-concavity-estimate-supercritical}.
\end{proof}

Similar to the mass-subcritical or critical case, to overcome the difficulty that $\mathcal{I}_c$ is unbounded from below on the $L^2$-sphere, we first reduce $\mathcal{I}_c$ to $\mathcal{O}_c^+$ via the following maximization problem:
\[
\rho(v)=\max_{u\in\mathcal M_c(v)}\mathcal I_c(u), \quad v\in \mathcal{O}_c^+.
\]
\begin{proposition}\label{prop:localReductionSuper}
Let $4<p<6$.  There exists $c_0>0$ such that, for every $c\ge c_0$ and every
$v\in \mathcal{O}_c^+$, there exists
\[
        \Phi_c(v)\in\mathcal M_c(v)
\]
such that
\[
        \mathcal I_c(\Phi_c(v))=\max_{u\in\mathcal M_c(v)}\mathcal I_c(u)=\rho_c(v).
\]
and 
\[
\dd\I_c(\Phi_c(v))[h]-2\omega\Re((\Phi_c(v), h)_{L^2}=0, \quad \forall h\in Y_c^-\oplus\operatorname{span}\{v\}.
\]
with $\omega=\omega(\Phi_c(v))\in \mathbb{R}^+$.
Writing
\[
        \Phi_c(v)=t_c(v)v+w_c(v),
        \qquad t_c(v)\ge0,
        \qquad w_c(v)\in Y_c^- ,
\]
and the reduced functional is defined as
\[
        \mathcal J_c^{\mathrm{loc}}(v):=\mathcal I_c(\Phi_c(v)),
        \qquad v\in \mathcal{O}_c^+,
\]
one has $t_c(v)>0$.  Moreover:
\begin{enumerate}[label=\textup{(\roman*)}]
\item the estimate
\begin{equation}\label{eq:fibre-basic-estimate-final-insert2}
        \mathcal A(\Phi_c(v))+\|\Phi_c(v)^-\|_c^2
        +\|\Phi_c(v)^-\|_{L^2}^2\|v\|_c^2
        \le \mathcal A(v)
\end{equation}
holds.  In particular,
\begin{equation}\label{eq:negative-part-bound-final-insert2}
        \|w_c(v)\|_c^2\le \mathcal A(v)=\frac{2a}{p}\int_K |v|^p\,dx;
\end{equation}
\item Up to a phase factor, the
maximizer is unique;
\item The map $\Phi_c:\mathcal{O}_c^+\to S_c$ and the
reduced functional $ \mathcal J_c^{\mathrm{loc}}(v)$
is of class $C^1$;
\item $u_n=\Phi_c(v_n)$ is a bounded Palais-Smale sequence of
 $\mathcal I_c|_{S_c}$ provided $v_n\in \mathcal{O}_c^+$ is a bounded Palais-Smale sequence
  of $\mathcal J_c^{\mathrm{loc}}$ on $\mathcal{O}_c^+$. Moreover, up to a subsequence, assume $\omega(v_n)\to \omega$,
  then $u_n=\Phi_c(v_n)$ is a bounded Palais-Smale sequence of the functional
  \[
  \I_c^\omega(u):=\I_c(u)-\omega\|u\|_{L^2}^2.
  \]
\end{enumerate}
\end{proposition}

\begin{proof}
      
  For $v\in\mathcal O_c^+$, Lemma~\ref{lem:graphPseudoGN}
gives
\begin{equation}\label{eq:I-v-positive-local-supercritical}
        \mathcal I_c(v)=\|v\|_c^2-\mathcal A(v)
        \ge (1-Cc^{-\eta})\|v\|_c^2
        \ge \frac12\|v\|_c^2>0 .
\end{equation}
Since $v\in\mathcal M_c(v)$, \eqref{eq:I-v-positive-local-supercritical} gives
$\rho_c(v)>0$.  Also, for $u=tv+w\in\mathcal M_c(v)$,
\[
        \mathcal I_c(u)=t^2\|v\|_c^2-\|w\|_c^2-\mathcal A(u)
        \le \|v\|_c^2,
\]
so $\rho_c(v)<+\infty$.
By Ekeland's variational principle there exists a maximizing Palais--Smale sequence
$(u_n)\subset\mathcal M_c(v)$ such that
\[
        \mathcal I_c(u_n)\to\rho_c(v),
        \qquad
        d(\mathcal I_c|_{\mathcal M_c(v)})(u_n)\to0 .
\]
Lemma~\ref{lem:local-fibre-PS-supercritical} yields, up to a subsequence, 
$u_n\to u$ in $Y_c$, $ \mathcal I_c(u)=\rho_c(v)$, and
\[
\dd\I_c(u)[h]-2\omega(v)\Re((u, h)_{L^2}=0, \quad \forall h\in Y_c^-\oplus\operatorname{span}\{v\}.
\]
for some $\omega(v)\in \mathbb R^+.$
 If $u=tv+w$ were such that $t=0$, then
$u\in Y_c^-$, $\|u\|_{L^2}=1$, and
\[
        \mathcal I_c(u)=-\|u\|_c^2-\mathcal A(u)<0,
\]
contradicting $\rho_c(v)>0$.  Therefore $t>0$.
 \begin{enumerate}[label=\textup{(\roman*)}]
                \item
Let $u=tv+w$ be any maximizer.  Comparing it with $v$ gives
\[
        t^2\|v\|_c^2-\|w\|_c^2-\mathcal A(u)
        \ge \|v\|_c^2-\mathcal A(v).
\]
Since $t^2=1-\|w\|_{L^2}^2$, we obtain
\begin{equation}\label{eq:basic-fibre-local-proof-supercritical}
        \mathcal A(u)+\|w\|_c^2+\|w\|_{L^2}^2\|v\|_c^2
        \le \mathcal A(v),
\end{equation}
which is \eqref{eq:fibre-basic-estimate-final-insert2}. 
\item
 Let $u_j=t_jv+w_j$, $j=1,2$, be two maximizers in
$\mathcal M_c(v)$.  By \eqref{eq:basic-fibre-local-proof-supercritical}, both
$w_1$ and $w_2$ belong to $\mathcal U_c(v)$.  For $\theta\in[0,1]$ set
\[
        w_\theta:=(1-\theta)w_1+\theta w_2 .
\]
The set $\mathcal U_c(v)$ is convex.  Thus $w_\theta\in\mathcal U_c(v)$.  If
$w_1\ne w_2$, Lemma~\ref{lem:local-fibre-concavity-supercritical} gives, for every
$\theta\in(0,1)$,
\[
        \mathcal F_v(w_\theta)
        >(1-\theta)\mathcal F_v(w_1)+\theta\mathcal F_v(w_2)
        =\rho_c(v),
\]
which contradicts the definition of $\rho_c(v)$.  Hence $w_1=w_2$, and then
$t_1=t_2$ because $t_j=(1-\|w_j\|_{L^2}^2)^{1/2}$.  This proves uniqueness on
$\mathcal M_c(v)$.
\item
We now prove the $C^1$ regularity.  Work locally on $Y_c^+\setminus\{0\}$ and write
$P(v)=v/\|v\|_{L^2}$.  For $\|w\|_{L^2}<1$, define
\[
        \mathfrak F(v,w)[\xi]
        :=\dd
        \mathcal I_c\big(a(w)P(v)+w\big)
       \left[ \dd a(w)[\xi]P(v)+\xi
        \right],
        \qquad \xi\in Y_c^- .
\]
 Fix
$v_0\in\mathcal O_c^+$ and write
\[
        \Phi_c(v_0)=a(w_0)v_0+w_0 .
\]
The derivative of $\mathfrak F$ with respect to $w$ at $(v_0,w_0)$ is the bilinear form
associated with $\dd^2\mathcal F_{v_0}(w_0)$.  By
Lemma~\ref{lem:local-fibre-concavity-supercritical},
\[
        -D_w\mathfrak F(v_0,w_0)[\xi][\xi]
        \ge \|\xi\|_c^2,
        \qquad \xi\in Y_c^- .
\]
The Lax--Milgram theorem shows that
$D_w\mathfrak F(v_0,w_0):Y_c^-\to(Y_c^-)^*$ is an isomorphism.  Hence the implicit
function theorem gives a $C^1$ map $v\mapsto w_c(v)$ in a neighborhood of $v_0$.  The
local maps agree on overlaps by uniqueness.  Therefore $v\mapsto w_c(v)$ is $C^1$ on
$\mathcal O_c^+$, and so are
\[
        t_c(v)=\bigl(1-\|w_c(v)\|_{L^2}^2\bigr)^{1/2},
        \qquad
        \Phi_c(v)=t_c(v)v+w_c(v).
\]
Thus $\mathcal J_c^{\mathrm{loc}}=\mathcal I_c\circ\Phi_c$ is $C^1$ on $\mathcal O_c^+$.

\item
Let \((v_n)\subset \mathcal O_c^+\) be a bounded Palais--Smale sequence for
\(\mathcal{J}_c^{\rm loc}\) on \(\mathcal O_c^+\), and set
\[
        u_n:=\Phi_c(v_n)=t_n v_n+w_n,\qquad
        t_n:=t_c(v_n),\quad w_n:=w_c(v_n)\in Y_c^- .
\]
We also write
\[
        \omega_n:=\omega(v_n).
\]
By \eqref{eq:fibre-basic-estimate-final-insert2} and Lemma \ref{lem:graphPseudoGN}(ii), we have
\[
        \|w_n\|_c^2+\|w_n\|_{L^2}^2\|v_n\|_c^2
        \le A(v_n)
        \le Cc^{-\eta}\|v_n\|_c^2 ,
\]
which yields \((u_n)\) is bounded in $Y_c$,
and for large $c$,
\[
        \|w_n\|_{L^2}^2\le Cc^{-\eta}<\frac12,
        \qquad
        t_n^2=1-\|w_n\|_{L^2}^2\ge \frac12 .
\]
In particular \(t_n\) is bounded away from zero. 
By (iii),
\[
\dd I_c(u_n)[h]
        -2\omega_n \Re\int_G u_n\cdot h\,dx=0,
        \qquad
        h\in Y_c^-\oplus {\rm span}\{v_n\}.
\]
Define 
\[
        R_n(h):=
        \dd I_c(u_n)[h]
        -2\omega_n \Re\int_G u_n\cdot h\,dx,
        \qquad h\in Y_c .
\]
Then
\[
        R_n(h)=0,
        \qquad
        h\in Y_c^-\oplus {\rm span}\{v_n\}.
\]
We claim that
\[
        \|R_n\|_{Y_c^*}\longrightarrow 0 .
\]
Indeed, Let \(h\in Y_c\). Write \(h=h^++h^-\), with \(h^\pm=P_c^\pm h\), and set
\[
        \alpha_n:=\Re\int_G h^+\cdot v_n\,dx,
        \qquad
        \zeta_n:=h^+-\alpha_n v_n .
\]
Then, it is clear that
\[
        R_n(h)=R_n(\zeta_n).
\]
For \(\zeta_n\in T_{v_n}S_c^+\),
\[
        D\Phi_c(v_n)[\zeta_n]
        =
        t_n\zeta_n+\dot t_n[\zeta_n]v_n
        +Dw_c(v_n)[\zeta_n],
\]
where
\[
        \dot t_n[\zeta_n]v_n+Dw_c(v_n)[\zeta_n]
        \in Y_c^-\oplus {\rm span}\{v_n\}.
\]
Furthermore, since \(\Phi_c(\mathcal O_c^+)\subset S_c\),
\[
        \Re\int_\G u_n\cdot D\Phi_c(v_n)[\zeta_n]\,dx=0.
\]
Thus, using the definition of \(\mathcal{J}_c^{\rm loc}\) and the fact that \(R_n\) vanishes on
\(Y_c^-\oplus {\rm span}\{v_n\}\), we obtain
\[
\begin{aligned}
        d\mathcal{J}_c^{\rm loc}(v_n)[\zeta_n]
        &=
        \dd I_c(u_n)\left[D\Phi_c(v_n)[\zeta_n] \right] \\
        &=
        R_n\big(D\Phi_c(v_n)[\zeta_n]\big)                       \\
        &=
        t_n R_n(\zeta_n).
\end{aligned}
\]
Consequently,
\[
        |R_n(h)|
        =
        |R_n(\zeta_n)|
        \le
        t_n^{-1}
        \|d\mathcal{J}_c^{\rm loc}(v_n)\|_{(T_{v_n}S_c^+)^*}
        \|\zeta_n\|_c .
\]
Since \((v_n)\) is bounded in \(Y_c\), then
\[
\begin{aligned}
        \|\zeta_n\|_c
        &\le \|h^+\|_c+|\alpha_n|\|v_n\|_c                                      \\
        &\le \|h^+\|_c+\|h^+\|_{L^2}\|v_n\|_c                                      \\
        &\le
        \left(1+(mc^2)^{-1/2}\sup_n\|v_n\|_c\right)\|h\|_c .
\end{aligned}
\]
Together with \(t_n\ge 1/\sqrt2\) and
\[
        \|d\mathcal{J}_c^{\rm loc}(v_n)\|_{(T_{v_n}S_c^+)^*}\longrightarrow 0 ,
\]
this proves
\[
        \|R_n\|_{Y_c^*}\longrightarrow 0 .
\]
Now let \(\xi\in T_{u_n}S_c\). Since
\[
        \Re\int_\G u_n\cdot \xi\,dx=0,
\]
we have
\[
        \ \dd\I_c|_{S_c}(u_n)[\xi]
        =
         \dd\I_c(u_n)[\xi]
        =
        R_n(\xi).
\]
Hence
\[
        \|\dd\I_c|_{S_c}(u_n)\|_{(T_{u_n}S_c)^*}
        \le
        \|R_n\|_{Y_c^*}
        \longrightarrow 0 .
\]
 This shows
that \((u_n)\) is a bounded Palais-Smale sequence for \(\I_c|_{S_c}\).
 Passing to a subsequence,
we may suppose that
\[
        \omega_n\to \omega .
\]
For every \(h\in Y_c\),
\[
\begin{aligned}
        \big|\dd\I_c^\omega(u_n)[h]\big|
        &=
        \left|
         \dd\I_c(u_n)[h]
        -2\omega \Re\int_\G u_n\cdot h\,dx
        \right|                                                   \\
        &\le
        |R_n(h)|
        +2|\omega_n-\omega|\,\|u_n\|_{L^2}\|h\|_{L^2}                       \\
        &\le
        \left(
        \|R_n\|_{Y_c^*}
        +C|\omega_n-\omega|
        \right)\|h\|_c .
\end{aligned}
\]
Therefore
\[
        \dd\I_c^\omega(u_n)\to 0
        \quad\text{in }Y_c^* .
\]
This proves (iv).

\end{enumerate}
\end{proof}

}

\begin{lemma}\label{lem:belowRestLocalExcess}
Let $4<p<6$.  There exist $C>0$ and $c_0>0$ such that, for every $c\ge c_0$ and every $v\in\mathcal O_c^+$,
\begin{equation}\label{eq:belowRestLocalAssumption}
\mathcal \mathcal{J}_c^{\rm loc}(v)<mc^2
\end{equation}
implies
\begin{equation}\label{eq:belowRestLocalExcess}
T_c(v)=\|v\|_c^2-mc^2\le C.
\end{equation}
Moreover, 
$$
\|v\|_{H^{1/2}}\leq C.
$$
\end{lemma}

\begin{proof}
Since 
\[
\mathcal \mathcal{J}_c^{\rm loc}(v)\ge \mathcal I_c(v)
=mc^2+T_c(v)-\frac{2a}{p}\int_\K |v|^p\dd x.
\]
Combining this inequality with \eqref{eq:belowRestLocalAssumption} and \eqref{eq:localModifiedGraphGN} gives
\[
0>T_c(v)-C\bigl(1+T_c(v)^\theta+c^{-\eta}T_c(v)\bigr).
\]
For $c$ large, $1-Cc^{-\eta}\ge1/2$.  Since $\theta<1$, then $T_c(v)\le C$, with $C$ independent of $c$.
By Lemma \ref{lemma2.1}, we get
$$
\|v\|_{H^{1/2}}\leq C.
$$
\end{proof}

Because $\mathcal \mathcal{J}_c^{\rm loc}$ is only defined on $\mathcal O_c^+$, we extend it by truncation. Let $\tau\in C^\infty([0,+\infty),[0,1])$ be nonincreasing, with
\[
\tau(t)=1\quad(0\le t\le1/2),
\qquad
\tau(t)=0\quad(t\ge3/4),
\]
and set $\tau_c(r)=\tau(r/c^s)$. For $v\in S_c^+$ define
\[
\Phi_{c,T}(v)=
\begin{cases}
\tau_c(\|v\|_c)\Phi_c(v),& \|v\|_c<c^s,\\[1mm]
0,& \|v\|_c\ge c^s,
\end{cases}
\]
where the first line is well defined by Proposition \ref{prop:localReductionSuper}. Since $\tau_c$ vanishes near $\|v\|_c=c^s$, this is a $C^1$ map. Put
\begin{equation}\label{eq:GammaTGraph}
\Gamma_{c,T}(v):=
\frac{\Phi_{c,T}(v)+(1-\tau_c(\|v\|_c))v}
{\|\Phi_{c,T}(v)+(1-\tau_c(\|v\|_c))v\|_{L^2}}
\end{equation}
and
\begin{equation}\label{eq:JTGraph}
\mathcal J_{c,T}(v)
:=\|\Gamma_{c,T}(v)^+\|_c^2-\|\Gamma_{c,T}(v)^-\|_c^2
-\frac{2a}{p}\tau_c(\|v\|_c)^2\int_\K |\Phi_c(v)|^p\dd x,
\end{equation}
with the convention that the last term is zero when $\|v\|_c\ge c^s$. Then $\mathcal J_{c,T}\in C^1(S_c^+,\R)$ is even and satisfies
\begin{equation}\label{eq:JTproperties}
\mathcal J_{c,T}(v)=\mathcal \mathcal{J}_c^{\rm loc}(v)
\quad\text{if }\|v\|_c<c^s/2,
\qquad
\mathcal J_{c,T}(v)=\|v\|_c^2
\quad\text{if }\|v\|_c\ge3c^s/4.
\end{equation}

Similar to Lemma \ref{lem:FWgraph}, we can conclude the following result.

\begin{lemma}\label{lem:FWgraph2}
Let $A\subset S_\infty$ be a compact set, we define
\begin{equation}
\iota(f):=\begin{pmatrix}f\\0\end{pmatrix},
\qquad f\in A,
\end{equation}
and
\begin{equation}
z_c(f):=P_c^+\iota(f),
\qquad
\Theta_c(f):=\frac{z_c(f)}{\|z_c(f)\|_{L^2}}\in S_c^+ .
\end{equation}
Then, 
\begin{equation}
\mathcal J_{c, T}(\Theta_c(f))
\le mc^2+\mathcal I_\infty(f)+o_c(1)
\qquad\text{uniformly for }f\in A.
\end{equation}
\end{lemma}

\begin{lemma}\label{lem:truncLocalizationGraph} There exist constants $C>0$ and $c_0>0$ such that, if $c\ge c_0$ and
\begin{equation}\label{eq:JTbelowRest}
\mathcal J_{c,T}(v)<mc^2,
\qquad v\in S_c^+,
\end{equation}
then
\begin{equation}\label{eq:truncInactiveConclusion}
\|v\|_c<c^s/2,
\qquad
T_c(v)=\|v\|_c^2-mc^2\le C.
\end{equation}
  Moreover, $\mathcal J_{c,T}$ satisfies the Palais--Smale condition at every level $d<mc^2$, and every critical point of $\mathcal J_{c,T}$ at such a level is a critical point of $\mathcal \mathcal{J}_c^{\rm loc}$.
\end{lemma}

\begin{proof}

If $\|v\|_c\ge c^s$, then by the definition of the truncation,
\[
\mathcal J_{c,T}(v)=\|v\|_c^2\ge c^{2s}>mc^2
\]
for large $c$, because $s>1$.  Hence every $v$ satisfying \eqref{eq:JTbelowRest} must obey $\|v\|_c<c^s$.

Assume now, by contradiction, that $c^s/2\le\|v\|_c<c^s$. The truncation gives
\begin{equation}\label{eq:JTtransitionLowerRevised}
\mathcal J_{c,T}(v)
\ge \frac{\mathcal \mathcal{J}_c^{\rm loc}(v)\|v\|_c^2}
        {\mathcal \mathcal{J}_c^{\rm loc}(v)+\|v\|_c^2}
\ge \frac12\mathcal \mathcal{J}_c^{\rm loc}(v),
\end{equation}
where we used $\mathcal \mathcal{J}_c^{\rm loc}(v)\le\|v\|_c^2$. Therefore \eqref{eq:JTbelowRest} implies
\begin{equation}\label{eq:JlocTwoRest}
\mathcal \mathcal{J}_c^{\rm loc}(v)<2mc^2.
\end{equation}
Since $v\in\mathcal M_c(v)$, we have $\mathcal \mathcal{J}_c^{\rm loc}(v)\ge\mathcal I_c(v)$.  Thus \eqref{eq:JlocTwoRest} gives
\[
\mathcal I_c(v)-mc^2<mc^2.
\]
Writing $T=T_c(v)$ and using Lemma \ref{lem:graphPseudoGN},
\begin{align*}
\mathcal I_c(v)-mc^2
&=T-\frac{2a}{p}\int_\K |v|^p\dd x\\
&\ge T-C\left(1+T^\theta+c^{-\frac{p-2}{2}}T^{2\theta}\right).
\end{align*}
Because $\|v\|_c<c^s$, we have $T\le c^{2s}$ and therefore, as in \eqref{eq:localModifiedGraphGN},
\[
c^{-\frac{p-2}{2}}T^{2\theta}
\le c^{-\eta}T.
\]
Hence
\begin{equation}\label{eq:TBoundTransition}
(1-Cc^{-\eta})T-C T^\theta-C<mc^2.
\end{equation}
Since $\theta<1$, \eqref{eq:TBoundTransition} implies $T\le Cc^2$.  Consequently
\[
\|v\|_c^2=mc^2+T\le Cc^2,
\qquad\text{and hence}\qquad
\|v\|_c\le Cc.
\]
As $s>1$, for large $c$ this gives $\|v\|_c<c^s/2$, contradicting the transition assumption.  On this ball the truncation is inactive, so
\[
\mathcal J_{c,T}(v)=\mathcal \mathcal{J}_c^{\rm loc}(v)<mc^2.
\]
The excess bound in \eqref{eq:truncInactiveConclusion} follows from Lemma \ref{lem:belowRestLocalExcess}.

Let $(v_n)\subset S_c^+$ be a Palais--Smale sequence for $\mathcal J_{c,T}$ at a level $d<mc^2$.  The preceding localization gives, for all large $n$,
\[
\|v_n\|_c<c^s/2,
\qquad
\mathcal J_{c,T}(v_n)=\mathcal \mathcal{J}_c^{\rm loc}(v_n),
\qquad
T_c(v_n)\le C.
\]
Repeat the same argument in Lemma \ref{lem:PSbelow}, we obtain the compactness of $v_n$. 

\end{proof}

\begin{proof}[\textbf{Proof of Theorem \ref{thm:1.1}}]
For $2<p\le 4$, $j\in\N$ define
\begin{equation}\label{eq:GammaDirac}
\Gamma_{c,j}:=\{A\subset S_c^+:A\text{ is closed, symmetric, and }\gamma(A)\ge j\}
\end{equation}
and
\begin{equation}\label{eq:DiracMinimax}
d_{c,j}:=\inf_{A\in\Gamma_{c,j}}\sup_{v\in A}\mathcal J_c(v).
\end{equation}
 By the definition of $e_j^a$, for $\varepsilon>0$, choose $A_j\in\Gamma_{\infty,j}$ such that
\begin{equation}\label{eq:AjChoice}
\sup_{f\in A_j}\mathcal I_\infty(f)\le e_j^a+\varepsilon.
\end{equation}
 Then
\[
A_{c,j}:=\Theta_c(A_j)\in \Gamma_{c,j}.
\]
 Lemma \ref{lem:FWgraph} yields
\begin{align*}
d_{c,j}
&\le\sup_{v\in A_{c,j}}\mathcal J_c(v)+\varepsilon\\
&\le mc^2+\sup_{f\in A_j}\mathcal I_\infty(f)+\varepsilon+o_c(1)\\
&\le mc^2+e_j^a+\varepsilon+o_c(1).
\end{align*}
Combining this with
Lemma \eqref{lem:NLSnegative} gives
\[
d_{c,j}<mc^2,
\]
when $c$ is sufficiently large. Lemma \ref{lem:PSbelow} gives the Palais-Smale condition for $\mathcal J_c$ at each of these levels. 
The Minimax theorem, Theorem \ref{minimax}, therefore gives at
 least $j$ distinct pairs of critical points, $v_{c,1}, \cdots v_{c,j}$  of $\mathcal J_c$ on $S_c^+$.
  Proposition \ref{prop:reduction} shows that $u_{c,j}=\Phi_c(v_{c,j})$ are critical points of $\mathcal I_c|_{S_c}$. Thus there exist Lagrange multipliers $\omega_{c,j}\in (0,mc^2)$ such that
\[
\mathscr{D}_cu_{c,j}-\omega_{c,j}u_{c,j}=a\chi_\K|u_{c,j}|^{p-2}u_{c,j},
\qquad \|u_{c,j}\|_{L^2}^2=1.
\]
For mass-supercritical case, namely, $4<p<6$, repeat the same argument for $\mathcal{J}_{c,T}$, we obtain at
 least $j$ distinct pairs of critical points of $\mathcal J_{c,T}$ on $S_c^+$.
By Lemma \ref{lem:truncLocalizationGraph} and Proposition \ref{prop:localReductionSuper}, we can obtain the multiplicity 
results for normalized solutions to \eqref{dirac}.
\end{proof}

\section{Nonrelativistic Limit of the Obtained Solutions}\label{sec:NR}
      For each $j\in \N$, let $u_{c,j}=(f_{c,j},g_{c,j})^T$ be a sequence of normalized Dirac solutions obtained by Theorem \ref{thm:1.1}.

\begin{lemma}\label{H1b}
      $$
      \sup_{c>c_0}\|u_{c,j}\|_{H^1}<\infty.
      $$
\end{lemma}
\begin{proof}
Since
\begin{equation}\label{eq:DnormEstimate}
\|\D_c u_{c,j}\|_{L^2}^2=m^2c^4\|u_{c,j}\|_{L^2}^2+c^2\|u_{c,j}'\|_{L^2}^2
=\|\omega_{c,j}u_{c,j}+a\chi_\K|u_{c,j}|^{p-2}u_{c,j}\|_{L^2}^2.
\end{equation}
and $\|u_{c,j}\|_{L^2}=1$ and $\omega_{c,j}\in (0, mc^2)$,  we get
\begin{equation}\label{aa}
   c^2\|u_{c,j}'\|_{L^2}^2\le C\left(1+c^2\|u_{c,j}\|_{L^p(\K)}^p+\|u_{c,j}\|_{L^{2p-2}(\K)}^{2p-2}\right).     
\end{equation}
If $2<p\le 4$, Gagliardo--Nirenberg inequality implies that the right-hand side is bounded by $C+C\|u_{c,j}'\|_{L^2}^\theta$ with some $\theta\le 2$. Hence $(u_{c,j})$ is bounded in $H^1(\G,\C^2)$.
If $4<p<6$, Lemma \ref{lem:truncLocalizationGraph} implies 
     $$
      \sup_{c>c_0}\|u_{c,j}\|_{H^{1/2}}<\infty.
      $$
      Consequently, it follows from the Sobolev embedding $H^{1/2}(\G)\hookrightarrow L^q(\G)$ for each $q>2$ that
           $$
      \sup_{c>c_0}\|u_{c,j}\|_{L^q}<\infty.
      $$
      Combining this with \eqref{aa}, we get
        $$
      \sup_{c>c_0}\|u_{c,j}\|_{H^{1}}<\infty.
      $$
\end{proof}

\begin{lemma}\label{op}
For each $j\in \mathbb{N}$,
    $$
    -\infty < \liminf\limits_{c\to \infty} (\omega_{c,j}-mc^2)\leq  \limsup\limits_{c\to \infty} (\omega_{c,j}-mc^2)< 0.
    $$
\end{lemma}
\begin{proof}
    It follows from Lemma \ref{lem:FWgraph} and Lemma \ref{lem:FWgraph2} that $\limsup\limits_{c\to \infty} (\omega_{c,j}-mc^2)< 0$.
    On the other hand,  multiplying both sides of \eqref{dirac} by \(u_{c,j}^+\) and integrating, we obtain
  $$
  \|u_{c,j}^+\|_c^2 -\Re\int_{\K} |{u_{c,j}}|^{p-2} u_{c,j}\cdot u_{c,j}^+d x= \omega_{c,j}\|u_{c,j}^+\|_{L^2}^2.
  $$
  Using Lemma \ref{H1b}, we obtain that there exists a constant ${C}$ independent of $c$, such that
  $$
  \Re\int_{\K} |{u_{c,j}}|^{p-2} u_{c,j}\cdot u_{c,j}^+ dx\leq {C}.
  $$
  Hence, we get
  $$
  (\omega_{c,j}-mc^2)\|u_{c,j}^+\|_{L^2}^2 \geq \omega_{c,j}\|u_{c,j}^+\|_{L^2}^2- \|u_{c,j}^+\|_c^2\geq  -{C},
  $$
  combining this with \eqref{eq:fibre-basic-estimate-final-insert2}, we get $-\infty < \liminf\limits_{c\to \infty} (\omega_{c,j}-mc^2).$ 
This completes the proof.
\end{proof}

\begin{lemma}\label{gh1}
        $$
        \|g_{c,j}\|_{H^1}\le \frac{C}{c}.
        $$
\end{lemma}
\begin{proof}
               We write \eqref{dirac} componentwise:
\begin{equation}\label{eq:components1}
\begin{cases}
-ic g_{c,j}'+(mc^2-\omega_{c,j})f_{c,j}=a\chi_\K|u_{c,j}|^{p-2}f_{c,j},\\[1mm]
-ic f_{c,j}'-(mc^2+\omega_{c,j})g_{c,j}=a\chi_\K|u_{c,j}|^{p-2}g_{c,j}.
\end{cases}
\end{equation}
        From the second equation in \eqref{eq:components1},
        \begin{equation}\label{kl}
                g_{c,j}=-\frac{ic}{mc^2+\omega_{c,j}}f_{c,j}'
-\frac{a}{mc^2+\omega_{c,j}}\chi_\K|u_{c,j}|^{p-2}g_{c,j}.
        \end{equation}
Using the uniform $H^1$ bound, the embedding $H^1(\G)\hookrightarrow L^\infty(\G)$, we get
\[
\|g_{c,j}\|_{L^2}\le \frac{C}{c}.
\]
Similarly, by combining the first equation in \eqref{eq:components1} with Lemma \ref{op}, $\|g_{c,j}'\|_{L^2}\leq Cc^{-1}$.
Therefore
\begin{equation}\label{eq:gSmallNR1}
\|g_{c,j}\|_{H^1(\G)}\le \frac{C}{c}.
\end{equation}
\end{proof}
\begin{comment}

\begin{lemma}\label{H2b}
      $$
      \sup_{c>c_0}\|u_{c,j}\|_{H^2}<\infty.
      $$
\end{lemma}
\begin{proof}
       We write \eqref{dirac} componentwise:
\begin{equation}\label{eq:components1}
\begin{cases}
-ic g_{c,j}'+(mc^2-\omega_{c,j})f_{c,j}=a\chi_\K|u_{c,j}|^{p-2}f_{c,j},\\[1mm]
-ic f_{c,j}'-(mc^2+\omega_{c,j})g_{c,j}=a\chi_\K|u_{c,j}|^{p-2}g_{c,j}.
\end{cases}
\end{equation}
Then, combining \eqref{eq:components1}, Lemma \ref{op} and Lemma \ref{H1b}, for each $q\ge 2$, we have
\begin{equation}\label{w1p}
        \sup_{c>c_0}\|u_{c,j}\|_{W^{1,q}}<\infty.
\end{equation}
      Consequently, differentiating both sides of \eqref{eq:components1} and combining with
       \eqref{w1p}, we obtain
        $$
      \sup_{c>c_0}\|u_{c,j}\|_{H^2}<\infty.
      $$
\end{proof}
\end{comment}

\begin{proof}[\textbf{Proof of Theorem 1.2}]

Let
\[
\rho_c:=\|f_{c,j}\|_{L^2},
\qquad
\widetilde f_{c,j}:=\rho_c^{-1}f_{c,j}.
\]
By Lemma \ref{gh1} and $\|u_{c,j}\|_{L^2}=1$, $\rho_c\to1$ and $\widetilde f_{c,j}\in S_\infty$ for large $c$.
 We now show that $\widetilde f_{c,j}$ is a Palais-Smale 
 sequence for $\mathcal I_\infty$ on $S_\infty$
 at the negative level $d_\infty$. Eliminating $g_{c,j}$ from \eqref{eq:components1} gives, in $H^{-1}(\G)$,
\begin{equation}\label{eq:secondOrderApprox}
-\frac{c^2}{mc^2+\omega_{c,j}}f_{c,j}''+\nu_{c,j} f_{c,j}
=a\chi_\K|u_{c,j}|^{p-2}f_{c,j}+r_c,
\end{equation}
where $r_c\to0$ in $H^{-1}(\G)$. Indeed, all error 
terms contain either the factor $(mc^2+\omega_{c,j})^{-1}$
 multiplying the localized nonlinearity or the 
 component $g_{c,j}$, and hence vanish by the uniform
  $H^1$ bound and Lemma \ref{gh1}. Since for each $q>1$
\[
\frac{c^2}{mc^2+\omega_{c,j}}=\frac1{2m}+o_c(1),
\qquad
|u_{c,j}|^{p-2}f_{c,j}-|f_{c,j}|^{p-2}f_{c,j}=o_c(1)
\quad\text{in }L^{q}(\K),
\]
we obtain, uniformly for $\varphi\in H^1(\G)$ with $\|\varphi\|_{H^1}\le1$,
\begin{equation}\label{eq:limitPSderivative}
\frac1{2m}\Re \int_\G \widetilde f_{c,j}'\overline{\varphi'}\dd x
-a\Re \int_\K |\widetilde f_{c,j}|^{p-2}\widetilde f_{c,j}\overline{\varphi}\dd x
+\nu_{c,j}\Re \int_\G \widetilde f_{c,j}\overline{\varphi}\dd x=o_c(1).
\end{equation}
where $\nu_{c,j}:=mc^2-\omega_{c,j}$, and up to a subsequence, $\nu_{c,j}\to \nu_j>0$.
Hence, $\widetilde f_{c,j}$ is a Palais-Smale sequence for $\mathcal I_\infty|_{S_\infty}$, with Lagrange multiplier $-\nu_j$.
Since $g_{c,j}=-ic(2mc^2+o(c^2))^{-1}f_{c,j}'+o(c^{-1})$ in $L^2$, the identity $\mathcal I_c(u)=(\mathscr{D}_cu,u)_2-\frac{2a}{p}\int_\K |u|^p\dd x$, and an edgewise integration by parts, one obtains
\begin{equation}\label{eq:energyExpansionNR}
\mathcal I_c(u_{c,j})-mc^2
=\frac1{2m}\int_\G |f_{c,j}'|^2\dd x-\frac{2a}{p}\int_\K |f_{c,j}|^p\dd x+o_c(1)
=\mathcal I_\infty(\widetilde f_{c,j})+o_c(1),
\end{equation}
 Together with Lemma \ref{lem:FWgraph} and Lemma \ref{lem:FWgraph2}, this gives
\[
\mathcal I_\infty(\widetilde f_{c,j})\to d_\infty<0.
\]
Therefore $\widetilde f_{c,j}$ is a Palais--Smale sequence for $\mathcal I_\infty|_{S_\infty}$ at a negative level. By Lemma \ref{lem:NLSPS}, after extracting a subsequence,
\[
\widetilde f_{c,j}\to f_\infty\quad\text{strongly in }H^1(\G).
\]
Since $\rho_c\to1$, the same convergence holds for $f_{c,j}$.
\end{proof}

\appendix
\section{Laplacians and Dirac operators on noncompact metric graphs}
\label{app:graph-operators}

This appendix mainly introduces some
 basic properties of the Laplacian and Dirac operators 
 on noncompact metric graphs $\G=(V,E)$.
\subsection{The Laplacian operator}

Set
\[
H^1_{\mathrm{K}}(\mathcal{G})
:=
\big\{f\in H^1(\mathcal{G}):
 f_e(v)=f_h(v)\text{ for all }e,h\succ v,\ v\in V _{\mathcal{K}}
\big\},
\]
and
\[
H^1_{\mathrm{D}}(\mathcal{G})
:=
\left\{g\in H^1(\mathcal{G}):
\sum_{e\succ v} g_e(v)^{\pm}=0
\text{ for every }v\in V _{\mathcal{K}}
\right\}.
\]
The subscript \(\mathrm{K}\) refers to the Kirchhoff-type vertex conditions, 
while \(\mathrm{D}\) refers to the dual condition.
Let
\[
T:H^1_{\mathrm{K}}(\mathcal{G})\subset L^2(\mathcal{G})\longrightarrow L^2(\mathcal{G}),
\qquad
Tf=f'.
\]
The operator \(T\) is closed and densely defined.  Its adjoint is
\[
T^*g=-g',
\qquad
\operatorname{dom}(T^*)=H^1_{\mathrm{D}}(\mathcal{G}).
\]
Indeed, for \(f\in H^1_{\mathrm{K}}(\mathcal{G})\) and \(g\in H^1(\mathcal{G})\), integration by parts on every edge gives
\[
\sum_{e\in E }\int_{I_e} f'_e\overline{g_e}\,dx
=
-
\sum_{e\in E }\int_{I_e} f_e\overline{g'_e}\,dx
+
\sum_{v\in V _{\mathcal{K}}}
 f(v)
 \overline{\left(\sum_{e\succ v}g_e(v)^{\pm}\right)}.
\]
The boundary terms at infinity vanish for \(H^1\)-functions on half-lines.  Since \(f\) is continuous at each vertex in \(V _{\mathcal{K}}\), the number \(f(v)\) is well-defined, and these vertex values can be prescribed independently.  Hence the boundary term vanishes for all \(f\in H^1_{\mathrm{K}}(\mathcal{G})\) if and only if
\[
\sum_{e\succ v}g_e(v)^{\pm}=0,
\qquad v\in V _{\mathcal{K}},
\]
which proves the formula for \(T^*\).

The scalar Laplacian operator with Kirchhoff-type vertex condition is
\[
L_{\mathcal{G}}^{\mathrm{K}}:=T^*T.
\]
It is a nonnegative self-adjoint operator on \(L^2(\mathcal{G})\).  It acts edgewise as
\[
L_{\mathcal{G}}^{\mathrm{K}}f=-f'',  
\]
with domain
\[
\operatorname{dom}(L_{\mathcal{G}}^{\mathrm{K}})
=
\left\{
 f\in H^2(\mathcal{G}):
\begin{array}{l}
 f_e(v)=f_h(v),\quad e,h\succ v,\ v\in V  _{\mathcal{K}},\\[2mm]
 \displaystyle\sum_{e\succ v}\partial_e f(v)=0,
 \quad v\in V  _{\mathcal{K}}
\end{array}
\right\}.
\]
The associated closed quadratic form is
\[
q_{\mathcal{G}}[f]
=
\|(L_{\mathcal{G}}^{\mathrm{K}})^{1/2}f\|_{2}^{2}
=
\|f'\|_{2}^{2},
\qquad
\operatorname{dom}(q_{\mathcal{G}})=
\operatorname{dom}((L_{\mathcal{G}}^{\mathrm{K}})^{1/2})=H^1_{\mathrm{K}}(\mathcal{G}).
\]
Consequently,
\[
\|(I+L_{\mathcal{G}}^{\mathrm{K}})^{1/2}f\|_{L^2}^2
=
\|f\|_{L^2}^2+
\|f'\|_{L^2}^2,
\qquad f\in H^1_{\mathrm{K}}(\mathcal{G}).
\]
Thus the inhomogeneous form norm induced by \(I+L_{\mathcal{G}}^{\mathrm{K}}\) is exactly the usual edgewise \(H^1\)-norm restricted to the Kirchhoff form domain.

The Laplacian operator with dual Kirchhoff-type vertex condition is
\[
L_{\mathcal{G }}^{\mathrm{D}}:=TT^*.
\]
It is also a nonnegative self-adjoint operator on \(L^2(\mathcal{G})\).  It acts edgewise as
\[
L_{\mathcal{G }}^{\mathrm{D}}g=-g'',
\]
with domain
\[
\operatorname{dom}(L_{\mathcal{G }}^{\mathrm{D}})
=
\left\{
 g\in H^2(\mathcal{G}):
\begin{array}{l}
 \displaystyle\sum_{e\succ v}g_e(v)^{\pm}=0,
 \quad v\in V  _{\mathcal{K}},\\[2mm]
 g'_e(v)=g'_h(v),\quad e,h\succ v,\ v\in V  _{\mathcal{K}}
\end{array}
\right\}.
\]
Here \(g'_e(v)\) denotes the derivative with respect to the fixed coordinate on the edge \(e\), not the outgoing derivative.  Its form domain is
\[
\operatorname{dom}\big((L_{\mathcal{G }}^{\mathrm{D}})^{1/2}\big)=H^1_{\mathrm{D}}(\mathcal{G}),
\qquad
\|(L_{\mathcal{G }}^{\mathrm{D}})^{1/2}g\|_{L^2}^2
=
\|g'\|_{L^2}^2.
\]
Therefore,
\[
\|(I+L_{\mathcal{G }}^{\mathrm{D}})^{1/2}g\|_{L^2}^2
=
\|g\|_{L^2}^2+
\|g'\|_{L^2}^2,
\qquad g\in H^1_{\mathrm{D}}(\mathcal{G}).
\]
%The inhomogeneous norm induced by \(I+L_{\mathcal{G }}^{\mathrm{D}}\) is thus equivalent, in fact equal, to the usual edgewise \(H^1\)-norm on \(H^1_{\mathrm{D}}(\mathcal{G})\).  By contrast, the homogeneous seminorm
%\[
%\|(L_{\mathcal{G }}^{\mathrm{D}})^{1/2}g\|_{L^2}=
%\|g'\|_{L^2}
%\]
%is not equivalent to the full \(H^1\)-norm in general.  There are two related reasons: \(L_{\mathcal{G }}^{\mathrm{D}}\) may have a nontrivial kernel, and \(0\) lies at the bottom of the essential spectrum on every noncompact graph with a half-line.
Since
\[
L_{\mathcal{G}}^{\mathrm{K}}=T^*T,
\qquad
L_{\mathcal{G }}^{\mathrm{D}}=TT^*,
\]
 their positive spectral parts are unitarily equivalent through the polar decomposition of \(T\).  In particular, away from zero their spectra agree, with the same multiplicities for isolated positive eigenvalues.

\begin{lemma}
\label{lem:laplacian-spectra}
One has
\[
\sigma(L_{\mathcal{G}}^{\mathrm{K}})=\sigma_{\mathrm{ess}}(L_{\mathcal{G}}^{\mathrm{K}})=[0,+\infty),
\qquad
\sigma(L_{\mathcal{G }}^{\mathrm{D}})=\sigma_{\mathrm{ess}}(L_{\mathcal{G }}^{\mathrm{D}})=[0,+\infty).
\]
\end{lemma}

\begin{proof}
Both operators are nonnegative, hence their spectra are contained in \([0,+\infty)\).  Conversely, let \(\mu\ge 0\).  Since \(\mathcal{G}\) is noncompact and has finitely many edges, it contains at least one half-line, say \(h=[0,+\infty)\).  Choose \(\eta\in C_c^{\infty}(0,1)\) with \(\|\eta\|_{L^2(0,1)}=1\), and define
\[
\eta_n(x):=n^{-1/2}\eta\!\left(\frac{x-n^2}{n}\right),
\qquad x\in[0,+\infty).
\]
Then \(\eta_n\) is supported in \((n^2,n^2+n)\), hence away from all vertices.  Let \(k=\sqrt{\mu}\), and set
\[
\phi_n(x):=\eta_n(x)e^{ikx}
\]
on the half-line \(h\), while \(\phi_n=0\) on every other edge.  Since \(\phi_n\) vanishes in a neighborhood of all vertices,
\[
\phi_n\in\operatorname{dom}(L_{\mathcal{G}}^{\mathrm{K}})
\cap
\operatorname{dom}(L_{\mathcal{G }}^{\mathrm{D}}).
\]
Moreover, \(\|\phi_n\|_{L^2}=1\) and \(\phi_n\rightharpoonup0\) weakly in \(L^2(\mathcal{G})\).  A direct computation gives
\[
\|(L_{\mathcal{G}}^{\mathrm{K}}-\mu)\phi_n\|_{L^2}\to0,
\qquad
\|(L_{\mathcal{G }}^{\mathrm{D}}-\mu)\phi_n\|_{L^2}\to0.
\]
By Weyl's criterion,
\[
\mu\in\sigma_{\mathrm{ess}}(L_{\mathcal{G}}^{\mathrm{K}})
\cap
\sigma_{\mathrm{ess}}(L_{\mathcal{G }}^{\mathrm{D}}).
\]
Since \(\mu\ge0\) was arbitrary, \([0,+\infty)\) is contained in both essential spectra.  Together with nonnegativity, this proves the claim.
\end{proof}

The kernel of $\ker L_{\mathcal{G}}^{\mathrm{K}}$ is trivial.  Indeed,
\[
\ker L_{\mathcal{G}}^{\mathrm{K}}=
\ker T.
\]
Thus \(f\in\ker L_{\mathcal{G}}^{\mathrm{K}}\) if and only if \(f'_e=0\) on every edge and \(f\) is continuous at every vertex in \(V _{\mathcal{K}}\).  Since \(\mathcal{G}\) is connected and noncompact, any such \(L^2\)-function must vanish on each half-line, and therefore, by continuity and connectedness, vanish everywhere.  Hence
\[
\ker L_{\mathcal{G}}^{\mathrm{K}}=\{0\}.
\]

The kernel of $\ker L_{\mathcal{G}}^{\mathrm{D}}$ is
\[
\ker L_{\mathcal{G }}^{\mathrm{D}}=
\ker T^*.
\]
Hence \(g\in\ker L_{\mathcal{G }}^{\mathrm{D}}\) if and only if
\[
g'_e=0\quad\text{on every edge }e,
\qquad
\sum_{e\succ v}g_e(v)^{\pm}=0
\quad\text{for every }v\in V  _{\mathcal{K}}.
\]
The \(L^2\)-condition forces \(g\) to vanish on every half-line.  On bounded edges, \(g\) is constant on each edge.  Therefore \(\ker L_{\mathcal{G }}^{\mathrm{D}}\) is the finite-dimensional space of edge constants \((a_e)_{e\in E _{\mathcal{K}}}\) satisfying the signed balance equations
\[
\sum_{e:e_-=v}a_e-
\sum_{e:e_+=v}a_e=0,
\qquad v\in V  _{\mathcal{K}},
\]
where loops contribute once with sign \(+1\) and once with sign \(-1\).

Let \(B\) be the oriented incidence matrix of the finite compact graph \(\mathcal{K}\):
\[
B_{v e}:=
\begin{cases}
1, & v=e_-,\\
-1, & v=e_+,\\
0, & \text{otherwise},
\end{cases}
\qquad
v\in V  _{\mathcal{K}},\ e\in E _{\mathcal{K}}.
\]
Then the preceding balance condition is exactly \(Ba=0\).  If \(\kappa(\mathcal{K})\) denotes the number of connected components of \(\mathcal{K}\), then
\[
\operatorname{rank}B=|V _{\mathcal{K}}|-\kappa(\mathcal{K}),
\]
and therefore
\[
\dim\ker L_{\mathcal{G }}^{\mathrm{D}}
=
| E _{\mathcal{K}}|-|V _{\mathcal{K}}|+\kappa(\mathcal{K})
=:b_1(\mathcal{K}).
\]
Thus \(\ker L_{\mathcal{G }}^{\mathrm{D}}\) has dimension equal to the first Betti number of the compact core.  In particular, if \(\mathcal{K}\) has no cycle, then
\[
\ker L_{\mathcal{G }}^{\mathrm{D}}=\{0\}.
\]

\subsection{The Dirac operators}
\label{app:dirac-selfadjoint}
The massless Dirac operator is
\[
\mathscr{D}_0=-i\sigma_1\frac{d}{dx}.
\]
With respect to the decomposition
\[
L^2(\mathcal{G},\mathbb{C}^2)=L^2(\mathcal{G})\oplus L^2(\mathcal{G}),
\]
it is represented by the block operator
\[
\mathscr{D}_0
=
\begin{pmatrix}
0&iT^*\\
-iT&0
\end{pmatrix},
\qquad
\operatorname{dom}(\mathscr{D}_0)
=
\operatorname{dom}(T)\oplus\operatorname{dom}(T^*)
=H^1_{\mathrm{K}}(\mathcal{G})\oplus H^1_{\mathrm{D}}(\mathcal{G}).
\]
Equivalently, for \(u=(u^1,u^2)^T\),
\[
\operatorname{dom}(\mathscr{D}_0)
=
\left\{
 u\in H^1(\mathcal{G},\mathbb{C}^2):
\begin{array}{l}
 u^1_e(v)=u^1_h(v),\quad e,h\succ v,\ v\in V  _{\mathcal{K}},\\[2mm]
 \displaystyle\sum_{e\succ v}u^2_e(v)^{\pm}=0,
 \quad v\in V  _{\mathcal{K}}
\end{array}
\right\}.
\]
Since \(T\) is closed and densely defined, the operator
\[
\begin{pmatrix}
0&iT^*\\
-iT&0
\end{pmatrix}
\]
is self-adjoint on \(\operatorname{dom}(T)\oplus\operatorname{dom}(T^*)\).  Hence \(\mathscr{D}_0\) is self-adjoint on \(L^2(\mathcal{G},\mathbb{C}^2)\).

For \(c>0\) and \(m>0\), the massive Dirac operator
\[
\mathscr{D}_c
=
-ic\sigma_1\frac{d}{dx}+mc^2\sigma_3
=
c\mathscr{D}_0+mc^2\sigma_3,
\]
with the same vertex conditions, namely
\[
\operatorname{dom}(\mathscr{D}_c)=\operatorname{dom}(\mathscr{D}_0).
\]
The mass term \(mc^2\sigma_3\) is bounded and self-adjoint.  Therefore, by the Kato--Rellich theorem, \(\mathscr{D}_c\) is self-adjoint on \(\operatorname{dom}(\mathscr{D}_0)\).

The square of \(\mathscr{D}_0\) is defined on its natural domain
\[
\operatorname{dom}(\mathscr{D}_0^2)
=
\{u\in\operatorname{dom}(\mathscr{D}_0):\mathscr{D}_0u\in\operatorname{dom}(\mathscr{D}_0)\}.
\]
For \(u=(f,g)^T\),
\[
\mathscr{D}_0(f,g)^T=(iT^*g,-iTf)^T,
\]
and hence
\[
\mathscr{D}_0^2(f,g)^T
=(T^*Tf,TT^*g)^T
=(L_{\mathcal{G}}^{\mathrm{K}}f,L_{\mathcal{G }}^{\mathrm{D}}g)^T.
\]
Thus
\[
\mathscr{D}_0^2=L_{\mathcal{G}}^{\mathrm{K}}\oplus L_{\mathcal{G }}^{\mathrm{D}},
\qquad
\operatorname{dom}(\mathscr{D}_0^2)
=
\operatorname{dom}(L_{\mathcal{G}}^{\mathrm{K}})
\oplus
\operatorname{dom}(L_{\mathcal{G }}^{\mathrm{D}}).
\]
In particular, \(\mathscr{D}_0^2\) acts edgewise as \(-d^2/dx^2\) on both components.
  The first spinorial component satisfies the Kirchhoff-type vertex conditions, 
  while the second component satisfies the dual vertex conditions.

Since
\[
\sigma_3\operatorname{dom}(\mathscr{D}_0)=\operatorname{dom}(\mathscr{D}_0),
\qquad
\mathscr{D}_0\sigma_3=-\sigma_3\mathscr{D}_0,
\]
therefore
\[
\mathscr{D}_c^2
=
(c\mathscr{D}_0+mc^2\sigma_3)^2
=
c^2\mathscr{D}_0^2+m^2c^4I
\]
on
\[
\operatorname{dom}(\mathscr{D}_c^2)=\operatorname{dom}(\mathscr{D}_0^2).
\]
Consequently, as nonnegative self-adjoint operators,
\[
|\mathscr{D}_c|
=(\mathscr{D}_c^2)^{1/2}
=
\big(c^2\mathscr{D}_0^2+m^2c^4I\big)^{1/2}.
\]
with domains
\[
\operatorname{dom}(|\mathscr{D}_c|)=\operatorname{dom}(\mathscr{D}_c).
\]

\begin{proposition}[Spectrum of \(\mathscr{D}_c\)]

For every \(c>0\),
\[
\sigma(\mathscr{D}_c)=\sigma_{\mathrm{ess}}(\mathscr{D}_c)
=(-\infty,-mc^2]\cup[mc^2,+\infty).
\]
Consequently,
\[
\sigma(|\mathscr{D}_c|)=[mc^2,+\infty).
\]
\end{proposition}

\begin{proof}
From the identity
\[
\mathscr{D}_c^2=c^2\mathscr{D}_0^2+m^2c^4I
\]
and from
\[
\mathscr{D}_0^2=L_{\mathcal{G}}^{\mathrm{K}}\oplus L_{\mathcal{G }}^{\mathrm{D}},
\]
Lemma~\ref{lem:laplacian-spectra} gives
\[
\sigma(\mathscr{D}_0^2)=[0,+\infty).
\]
Hence
\[
\sigma(\mathscr{D}_c^2)=[m^2c^4,+\infty).
\]
By the spectral theorem,
\[
\sigma(\mathscr{D}_c)
\subset
(-\infty,-mc^2]\cup[mc^2,+\infty).
\]

It remains to prove the reverse inclusion.  Let
\[
\lambda\in(-\infty,-mc^2)\cup(mc^2,+\infty).
\]
Choose a half-line \(h=[0,+\infty)\) in \(\mathcal{G}\).  Let \(\eta\in C_c^{\infty}(0,1)\), \(\|\eta\|_{L^2(0,1)}=1\), and set
\[
\eta_n(x):=n^{-1/2}\eta\!\left(\frac{x-n^2}{n}\right).
\]
Define
\[
k_{\lambda}:=\frac{\sqrt{\lambda^2-m^2c^4}}{c}>0.
\]
Choose \(a_{\lambda}\in\mathbb{C}^2\), \(|a_{\lambda}|=1\), such that
\[
(ck_{\lambda}\sigma_1+mc^2\sigma_3)a_{\lambda}=\lambda a_{\lambda}.
\]
For example, one may take
\[
a_{\lambda}
=
\frac{1}{N_{\lambda}}
\begin{pmatrix}
1\\[1mm]
\displaystyle\frac{\lambda-mc^2}{ck_{\lambda}}
\end{pmatrix},
\]
with \(N_{\lambda}>0\) chosen so that \(|a_{\lambda}|=1\).

Set
\[
u_n(x):=\eta_n(x)e^{ik_{\lambda}x}a_{\lambda}
\]
on the half-line \(h\), and \(u_n=0\) on all other edges.  Since \(u_n\) vanishes in a neighborhood of all vertices, \(u_n\in\operatorname{dom}(\mathscr{D}_c)\).  Moreover,
\[
\|u_n\|_{L^2(\mathcal{G},\mathbb{C}^2)}=1,
\qquad
u_n\rightharpoonup0
\quad\text{weakly in }L^2(\mathcal{G},\mathbb{C}^2).
\]
A direct computation gives
\[
(\mathscr{D}_c-\lambda)u_n
=-ic\eta_n'(x)e^{ik_{\lambda}x}\sigma_1a_{\lambda},
\]
and therefore
\[
\|(\mathscr{D}_c-\lambda)u_n\|_{L^2}
\le c\|\eta_n'\|_{L^2(0,+\infty)}
=O(n^{-1})\to0.
\]
By Weyl's criterion,
\[
\lambda\in\sigma_{\mathrm{ess}}(\mathscr{D}_c).
\]

At the positive threshold \(\lambda=mc^2\), take
\[
a_+:=\begin{pmatrix}1\\0\end{pmatrix},
\qquad
u_n(x):=\eta_n(x)a_+
\]
on the half-line and zero elsewhere.  Then
\[
(mc^2\sigma_3-mc^2I)a_+=0,
\]
and
\[
(\mathscr{D}_c-mc^2)u_n
=-ic\eta_n'(x)\sigma_1a_+\to0
\quad\text{in }L^2.
\]
Thus \(mc^2\in\sigma_{\mathrm{ess}}(\mathscr{D}_c)\).  Similarly, for \(\lambda=-mc^2\), take
\[
a_-:=\begin{pmatrix}0\\1\end{pmatrix},
\qquad
u_n(x):=\eta_n(x)a_-.
\]
Then
\[
(mc^2\sigma_3+mc^2I)a_-=0,
\]
and
\[
(\mathscr{D}_c+mc^2)u_n
=-ic\eta_n'(x)\sigma_1a_-\to0
\quad\text{in }L^2.
\]
Therefore \(-mc^2\in\sigma_{\mathrm{ess}}(\mathscr{D}_c)\) and
\[
\sigma_{\mathrm{ess}}(\mathscr{D}_c)=(-\infty,-mc^2]\cup[mc^2,+\infty).
\]
 The statements for \(|\mathscr{D}_c|\) follows from the spectral theorem.
\end{proof}

The form domain of \(\mathscr{D}_c\) is
\[
Y_c=\operatorname{dom}(|\mathscr{D}_c|^{1/2}),
\]
 with inner product
\[
(u,v)_c
:=
\Re 
\big(|\mathscr D_c|^{1/2}u,|\mathscr D_c|^{1/2}v\big)_{L^2},
\qquad
\|u\|_c^2=(u,u)_c.
\]
The spectral gap gives
\[
\|u\|_c^2
\ge
mc^2\|u\|_{L^2}^2,
\qquad u\in Y_c.
\]
Since \(0\notin\sigma(\mathscr{D}_c)\), the positive and negative spectral projectors are
\[
P_c^+:=\mathbf{1}_{[mc^2,+\infty)}(\mathscr{D}_c),
\qquad
P_c^-:=\mathbf{1}_{(-\infty,-mc^2]}(\mathscr{D}_c),
\]
and
\[
P_c^\pm=\frac12\big(I\pm\mathscr{D}_c|\mathscr{D}_c|^{-1}\big).
\]
The spectral decomposition gives
\[
L^2(\mathcal{G},\mathbb{C}^2)=L_c^+\oplus L_c^-,
\qquad
L_c^\pm:=P_c^\pm L^2(\mathcal{G},\mathbb{C}^2),
\]
and
\[
Y_c=Y_c^+\oplus Y_c^-,
\qquad
Y_c^\pm:=P_c^\pm Y_c.
\]
The splitting is orthogonal both in \(L^2\) and for the inner product in $Y_c$.

Since
\[
|\mathscr{D}_c|^{1/2}
=
\big(c^2\mathscr{D}_0^2+m^2c^4I\big)^{1/4},
\]
we have, as sets,
\[
Y_c
=
\operatorname{dom}\big((c^2\mathscr{D}_0^2+m^2c^4I)^{1/4}\big)
=
\operatorname{dom}\big((I+\mathscr{D}_0^2)^{1/4}\big).
\]
Thus the underlying vector space \(Y_c\) is independent of \(c\), although the norm \(\|\cdot\|_c\) depends on \(c\).

For \(s\ge0\), introduce the  Hilbert space
\[
X^s:=\operatorname{dom}\big((I+\mathscr{D}_0^2)^{s/2}\big),
\qquad
\|u\|_{X^s}:=\|(I+\mathscr{D}_0^2)^{s/2}u\|_{L^2}.
\]
Then
\[
Y_c=X^{1/2}.
\]
Because
\[
\mathscr{D}_0^2=L_{\mathcal{G}}^{\mathrm{K}}\oplus L_{\mathcal{G }}^{\mathrm{D}},
\]
one has
\[
X^{1/2}
=
\operatorname{dom}\big((I+L_{\mathcal{G}}^{\mathrm{K}})^{1/4}\big)
\oplus
\operatorname{dom}\big((I+L_{\mathcal{G }}^{\mathrm{D}})^{1/4}\big),
\]
and, for \(u=(u^1,u^2)^T\),
\[
\|u\|_{X^{1/2}}^2
=
\|(I+L_{\mathcal{G}}^{\mathrm{K}})^{1/4}u^1\|_{L^2}^2
+
\|(I+L_{\mathcal{G }}^{\mathrm{D}})^{1/4}u^2\|_{L^2}^2.
\]
The fractional domains admit the interpolation descriptions
\[
\operatorname{dom}\big((I+L_{\mathcal{G}}^{\mathrm{K}})^{1/4}\big)
=
\big[L^2(\mathcal{G}),H^1_{\mathrm{K}}(\mathcal{G})\big]_{\frac12},
\]
and
\[
\operatorname{dom}\big((I+L_{\mathcal{G }}^{\mathrm{D}})^{1/4}\big)
=
\big[L^2(\mathcal{G}),H^1_{\mathrm{D}}(\mathcal{G})\big]_{\frac12},
\]
with equivalent norms.  

Since
\[
H^1_{\mathrm{K}}(\mathcal{G})\hookrightarrow H^1(\mathcal{G}),
\qquad
H^1_{\mathrm{D}}(\mathcal{G})\hookrightarrow H^1(\mathcal{G}),
\]
continuously, interpolation gives
\[
X^{1/2}=Y_c\hookrightarrow H^{1/2}(\mathcal{G},\mathbb{C}^2)=\big[L^2(\mathcal{G}),H^1(\mathcal{G},\mathbb{C}^2)\big]_{\frac12}
\]
continuously.  Hence there exists \(C>0\), depending only on \(\mathcal{G}\), such that
\[
\|u\|_{H^{1/2}(\mathcal{G},\mathbb{C}^2)}
\le
C\|u\|_{X^{1/2}},
\qquad u\in X^{1/2}.
\]

We now compare the \(c\)-dependent norm $\|\cdot\|_c$ with the fixed \(X^{1/2}\)-norm.  Let \(\mu_u\) be the spectral measure of the nonnegative self-adjoint operator \(\mathscr{D}_0^2\) associated with \(u\).  Then
\[
\|u\|_c^2
=
\int_{[0,+\infty)}(c^2\lambda+m^2c^4)^{1/2}\,d\mu_u(\lambda),
\]
whereas
\[
\|u\|_{X^{1/2}}^2
=
\int_{[0,+\infty)}(1+\lambda)^{1/2}\,d\mu_u(\lambda).
\]
It follows from the fact
\[
c\min\{1,m\}(1+
\lambda)^{1/2}\le (c^2\lambda+m^2c^4)^{1/2}
\le
c^2\max\{1,m\}(1+\lambda)^{1/2}
\]
that
\[
C_m^{-1}c\|u\|_{X^{1/2}}^2
\le
\|u\|_c^2
\le
C_m c^2\|u\|_{X^{1/2}}^2,
\qquad u\in Y_c.
\]
Thus, for each fixed \(c>0\), the norms \(\|\cdot\|_c\) and \(\|\cdot\|_{X^{1/2}}\) are equivalent on \(Y_c\). 
Consequently,
\[
\|u\|_{H^{1/2}(\mathcal{G},\mathbb{C}^2)}^2
\le
C c^{-1}\|u\|_c^2,
\qquad u\in Y_c,
\quad c\ge1.
\]

\noindent {Pan Chen\\
School of Mathematical Sciences,\\
 Shanghai Jiao Tong University, Shanghai 200240, P.R. China
\\
e-mail: chenpan2020@amss.ac.cn }
\medskip
\\
\noindent {Qi Guo\\
School of Mathematics,\\
Renmin University of China, Beijing, 100872, P.R. China\\
e-mail: qguo@ruc.edu.cn}


\begin{thebibliography}{99}

\bibitem{AdamiSerraTilliCalcVar2015}
R. Adami, E. Serra and P. Tilli,
\newblock NLS ground states on graphs,
\newblock \emph{Calc. Var. Partial Differential Equations} \textbf{54} (2015), 743--761.

\bibitem{AdamiSerraTilliJFA2016}
R. Adami, E. Serra and P. Tilli,
\newblock Threshold phenomena and existence results for NLS ground states on metric graphs,
\newblock \emph{J. Funct. Anal.} \textbf{271} (2016), 201--223.

\bibitem{AdamiSerraTilliCMP2017}
R. Adami, E. Serra and P. Tilli,
\newblock Negative energy ground states for the $L^2$-critical NLSE on metric graphs,
\newblock \emph{Comm. Math. Phys.} \textbf{352} (2017), 387--406.

\bibitem{AdamiSerraTilliCalcVar2019}
R. Adami, E. Serra and P. Tilli,
\newblock Multiple positive bound states for the subcritical NLS equation on metric graphs,
\newblock \emph{Calc. Var. Partial Differential Equations} \textbf{58} (2019), Paper No. 5.

\bibitem{BerkolaikoKuchmentBook}
G. Berkolaiko and P. Kuchment,
\newblock \emph{Introduction to Quantum Graphs},
\newblock Mathematical Surveys and Monographs, Vol. 186, American Mathematical Society, Providence, RI, 2013.

\bibitem{BorrelliCarloneTentarelliSIAM2019}
W. Borrelli, R. Carlone and L. Tentarelli,
\newblock Nonlinear Dirac equation on graphs with localized nonlinearities: bound states and nonrelativistic limit,
\newblock \emph{SIAM J. Math. Anal.} \textbf{51} (2019), 1046--1081.

\bibitem{BorrelliCarloneTentarelliNote2019}
W. Borrelli, R. Carlone and L. Tentarelli,
\newblock A note on the Dirac operator with Kirchhoff-type vertex conditions on noncompact metric graphs,
\newblock preprint, HAL hal-02034789, 2019.

\bibitem{BorrelliCarloneTentarelliJDE2021}
W. Borrelli, R. Carlone and L. Tentarelli,
\newblock On the nonlinear Dirac equation on noncompact metric graphs,
\newblock \emph{J. Differential Equations} \textbf{278} (2021), 326--357.

\bibitem{BorthwickChangJeanjeanSoaveNonlinearity2023}
J. Borthwick, X. Chang, L. Jeanjean and N. Soave,
\newblock Normalized solutions of $L^2$-supercritical NLS equations on noncompact metric graphs with localized nonlinearities,
\newblock \emph{Nonlinearity} \textbf{36} (2023), 3776--3795.

\bibitem{BuffoniEstebanSere2006}
B. Buffoni, M. J. Esteban and E. S\'er\'e,
\newblock Normalized solutions to strongly indefinite semilinear equations,
\newblock \emph{Adv. Nonlinear Stud.} \textbf{6} (2006), 323--347.

\bibitem{ChenDingGuoWangCalcVar2024}
P. Chen, Y. Ding, Q. Guo and H.-Y. Wang,
\newblock Nonrelativistic limit of normalized solutions to a class of nonlinear Dirac equations,
\newblock \emph{Calc. Var. Partial Differential Equations} \textbf{63} (2024), Paper No. 90.

\bibitem{ChenDingGuo2025}
P. Chen, Y. Ding and Q. Guo,
\newblock Existence and nonrelativistic limit of ground states to nonlinear Dirac equations,
\newblock preprint, arXiv:2507.11203, 2025.

\bibitem{ChenGuoYuJGA2026}
P. Chen, Q. Guo and Y. Yu,
\newblock Limit behavior of multiple bound states of nonlinear Dirac equations,
\newblock \emph{J. Geom. Anal.} \textbf{36} (2026), Paper No. 131.

\bibitem{DingYuZhaoJGA2023}
Y. Ding, Y. Yu and F. Zhao,
\newblock $L^2$-normalized solitary wave solutions of a nonlinear Dirac equation,
\newblock \emph{J. Geom. Anal.} \textbf{33} (2023), Paper No. 69.

\bibitem{DovettaJDE2018}
S. Dovetta,
\newblock Existence of infinitely many stationary solutions of the $L^2$-subcritical and critical NLSE on compact metric graphs,
\newblock \emph{J. Differ. Equations} \textbf{264} (2018), 4806--4821.

\bibitem{DovettaSerraTilliAdvMath2020}
S. Dovetta, E. Serra and P. Tilli,
\newblock Uniqueness and non-uniqueness of prescribed mass NLS ground states on metric graphs,
\newblock \emph{Adv. Math.} \textbf{374} (2020), Article 107352.


\bibitem{EstebanSereCMP1995}
M. J. Esteban and E. S\'er\'e,
\newblock Stationary states of the nonlinear Dirac equation: a variational approach,
\newblock \emph{Comm. Math. Phys.} \textbf{171} (1995), 323--350.

\bibitem{JeanjeanLu2019}
L. Jeanjean and S. S. Lu,
\newblock Nonradial normalized solutions for nonlinear scalar field equations,
\newblock \emph{Nonlinearity} \textbf{32} (2019), 4942--4966.

\bibitem{HeJi2025}
Z. He and C. Ji,
\newblock Normalized solutions of nonlinear Dirac equations on noncompact metric graphs with localized nonlinearities,
\newblock arXiv:2505.15100, 2025.

\bibitem{Krasnoselskii1964}
M. A. Krasnosel'skii,
\newblock \emph{Topological Methods in the Theory of Nonlinear Integral Equations},
\newblock Macmillan, New York, 1964.


\bibitem{Noja2014}
D. Noja,
\newblock Nonlinear Schr\"odinger equation on graphs: recent results and open problems,
\newblock \emph{Philos. Trans. R. Soc. Lond. Ser. A Math. Phys. Eng. Sci.} \textbf{372} (2014), 20130002.

\bibitem{Rabinowitz1986}
P. H. Rabinowitz,
\newblock \emph{Minimax Methods in Critical Point Theory with Applications to Differential Equations},
\newblock CBMS Regional Conference Series in Mathematics, Vol. 65, American Mathematical Society, Providence, RI, 1986.

\bibitem{SerraTentarelliJDE2016}
E. Serra and L. Tentarelli,
\newblock Bound states of the NLS equation on metric graphs with localized nonlinearities,
\newblock \emph{J. Differential Equations} \textbf{260} (2016), 5627--5644.

\bibitem{SerraTentarelliNA2016}
E. Serra and L. Tentarelli,
\newblock On the lack of bound states for certain NLS equations on metric graphs,
\newblock \emph{Nonlinear Anal.} \textbf{145} (2016), 68--82.

\bibitem{Struwe2008}
M. Struwe,
\newblock \emph{Variational Methods. Applications to Nonlinear Partial Differential Equations and Hamiltonian Systems},
\newblock 4th ed., Springer, Berlin, 2008.

\bibitem{TentarelliJMAA2016}
L. Tentarelli,
\newblock NLS ground states on metric graphs with localized nonlinearities,
\newblock \emph{J. Math. Anal. Appl.} \textbf{433} (2016), 291--304.

```bibtex
\bibitem{CotiZelatiNolasco2025}
V. Coti Zelati and M. Nolasco,
\newblock Normalized solutions for a nonlinear Dirac equation,
\newblock \emph{J. Differ. Equations} \textbf{414} (2025), 746--772.

\bibitem{Nolasco2021}
M. Nolasco,
\newblock A normalized solitary wave solution of the Maxwell-Dirac equations,
\newblock \emph{Ann. Inst. Henri Poincar{\'e}, Anal. Non Lin{\'e}aire} \textbf{38} (2021), 1681--1702.

\bibitem{CotiZelatiNolasco2019}
V. Coti Zelati and M. Nolasco,
\newblock Ground state for the relativistic one electron atom in a self-generated electromagnetic field,
\newblock \emph{SIAM J. Math. Anal.} \textbf{51} (2019), 2206--2230.

\end{thebibliography}
\end{document}